\documentclass[12pt]{amsart}

\usepackage{amsmath,amsthm,amssymb}
\usepackage{tikz}
\usetikzlibrary{arrows.meta}

\usepackage{a4wide}
\usepackage{enumerate}
\usepackage{tasks}
\usepackage{multicol}
\newtheorem{theorem}{Theorem}[section]
\newtheorem{lemma}[theorem]{Lemma}
\newtheorem{proposition}[theorem]{Proposition}
\newtheorem{corollary}[theorem]{Corollary}

\newtheorem{theo}{Theorem}

\theoremstyle{remark}
\newtheorem{remark}[theorem]{Remark}

\newtheorem*{claim*}{Claim}

\newcommand{\C}{\ensuremath{\mathbb{C}}}
\newcommand{\R}{\ensuremath{\mathbb{R}}}

\newcommand{\g}[1]{\ensuremath{\mathfrak{#1}}}
\newcommand{\wt}[1]{\ensuremath{\widetilde{#1}}}

\DeclareMathOperator{\tr}{tr}

\DeclareMathOperator{\id}{Id}

\DeclareMathOperator{\ad}{ad}

\DeclareMathOperator{\spann}{span}

\DeclareMathOperator{\Ric}{Ric}
\DeclareMathOperator{\Der}{Der}
\newcommand{\Cl}{\ensuremath{\mathsf{Cl}}}
\newcommand{\Ss}{\ensuremath{\mathcal{S}}}

\renewcommand{\H}{\ensuremath{\mathbb{H}}}
\renewcommand{\mod}[1]{\ensuremath{\;(\mathrm{mod\;}#1)}}

\newcommand{\Spin}{\ensuremath{\mathsf{Spin}}}

\begin{document}
\title[]{
Codimension one Ricci soliton subgroups\\of solvable Iwasawa groups
}
\author[M.~Dom\'{\i}nguez-V\'{a}zquez]{Miguel Dom\'{\i}nguez-V\'{a}zquez}
\address{Department of Mathematics, Universidade de Santiago de Compostela, Spain.}
\email{miguel.dominguez@usc.es}
\author[V.~Sanmart\'in-L\'opez]{V\'ictor Sanmart\'in-L\'opez}
\address{Department of Mathematics, Katholieke Universiteit te Leuven, Belgium.}
\email{victor.sanmartinlopez@kuleuven.be}
\author[H.~Tamaru]{Hiroshi Tamaru}
\address{Department of Mathematics, Osaka City University, Japan.}
\email{tamaru@sci.osaka-cu.ac.jp}

\begin{abstract}
Recently, Jablonski proved that, to a large extent, 
a simply connected solvable Lie group endowed with a left-invariant Ricci soliton metric can be isometrically embedded into the solvable Iwasawa group of a non-compact symmetric space. 
Motivated by this result, we classify codimension one subgroups of the solvable Iwasawa groups of irreducible symmetric spaces of non-compact type whose induced metrics are Ricci solitons. 
We also obtain the classifications of codimension one Ricci soliton subgroups of Damek-Ricci spaces and generalized Heisenberg groups. 
\end{abstract}

\thanks{The first two authors have been supported by projects PID2019-105138GB-C21 (AEI/FEDER, Spain) and ED431C 2019/10, ED431F 2020/04 (Xunta de Galicia, Spain). The first author acknowledges support of the Ram\'{o}n y Cajal program of the Spanish State Research Agency.  
The third author was supported by JSPS KAKENHI Grant Number JP19K21831. 
This work was partially supported by the Fonds Wetenschappelijk Onderzoek – Vlaanderen (FWO), the Fonds de la Recherche Scientifique – FNRS under EOS Project No.~G0H4518N, and Osaka City University Advanced Mathematical Institute (MEXT Joint Usage/Research Center on Mathematics and Theoretical Physics JPMXP0619217849). 
} 

\subjclass[2010]{53C25, 22E25, 53C30, 53C35, 53C40}
\keywords{Ricci soliton, solvsoliton, nilsoliton, hypersurface, codimension one, generalized Heisenberg group, Damek-Ricci space, symmetric space, Iwasawa group}
\maketitle
\section{Introduction}

The investigation of homogeneous Einstein manifolds and, more recently, homogeneous Ricci solitons constitutes an important area of research in differential geometry. By their very nature, these geometric structures have typically been studied with Lie theoretic methods in combination with tools of intrinsic Riemannian geometry. In spite of outstanding recent progress (e.g.~\cite{Jab:gt,Jab:crelle,LafuenteLauret:jdg,Lauret:mathann,Lauret:crelle}), the current understanding is not yet complete, and the classification problem seems to be rather difficult. 
Several nice classification results have been obtained, 
but mainly in low dimensions (see for example \cite{Arroyo-Lafente1, Arroyo-Lafente2, Fernandez-Culma, Will}). 

Recently, Jablonski~\cite{Jab:arxiv} proved the following remarkable result, 
which lies at the intersection of Ado's and Nash's embedding theorems: 
\emph{every simply connected Ricci soliton solvmanifold (in particular, every Einstein solvmanifold) or 
	$2$-step nilpotent Lie group with left-invariant metric can be realized as a submanifold of a symmetric space}. 
Let us be more precise. 
Let $M\cong G/K$ be a symmetric space of non-compact type. The connected real semisimple Lie group $G$ admits an Iwasawa decomposition $G\cong KAN$. 
This allows us to identify, as Riemannian manifolds, $M$ with the solvable Iwasawa group 
$AN$ endowed with certain left-invariant metric (see~\S\ref{subsec:symmetric_prelim} for more details). 
Now, let $S$ be a simply connected solvable Lie group with a left-invariant metric (solvmanifold), 
and assume that $S$ is either completely solvable and Ricci soliton, or $2$-step nilpotent. Then, Jablonski proved that there is an $n\in\mathbb{N}$ and an injective Lie group homomorphism $\phi\colon S\to AN$ that is also an isometric embedding, where $AN$ is the solvable Iwasawa group of the symmetric space $\mathsf{SL}(n,\R)/\mathsf{SO}(n)$, endowed with its natural (unique up to scaling) left-invariant Einstein metric. Since any Ricci soliton solvmanifold (and in particular, any Einstein solvmanifold) is isometric to a completely solvable Lie group with a left-invariant metric~\cite{Jab:crelle}, one gets the result emphasized above. 
 
Therefore, these results ensure that some important families of homogeneous Ricci solitons always arise as submanifolds of the solvable Iwasawa group associated with a symmetric space of non-compact type. This opens up the possibility of addressing the investigation of \emph{homogeneous Ricci solitons from the viewpoint of extrinsic submanifold geometry in symmetric spaces}. In this sense, on the one hand one may aim to have some interesting families of examples, and on the other hand to investigate general properties of such submanifolds and, ideally, derive classification results. 

Regarding the first problem, there are two important collections of examples of Einstein submanifolds of symmetric spaces of non-compact type. One of them is made of the irreducible totally geodesic submanifolds. This is a set of somehow trivial examples from an intrinsic viewpoint, as these submanifolds are themselves symmetric spaces. But one has to emphasize that the classification problem of totally geodesic submanifolds in symmetric spaces is still outstanding (see for example \cite{BO:jdg} for a recent contribution).  
The second set of examples is due to 
the third author~\cite{Tamaru}, 
who proved that the solvable part $A_\Phi N_\Phi$ of a parabolic subgroup  of the (real semisimple) isometry group of a symmetric space of non-compact type $M$ is an Einstein solvmanifold (arising as an equivariant isometric embedding into the corresponding Iwasawa group $AN$). Here, $\Phi$ is an arbitrary subset of simple restricted roots of the symmetric space. These submanifolds are interesting in that they provide examples of Einstein solvmanifolds with nilradicals of arbitrary large degrees of nilpotency. Moreover, from an extrinsic point of view, they are minimal (and, except in some limit cases, not totally geodesic) submanifolds.

In view of the existence of general embedding results and some interesting families of examples, 
we propose to undertake the project of investigating homogeneous Ricci solitons from the perspective of extrinsic submanifold geometry. 
In the present article, we will restrict our attention to submanifolds of codimension one. 
The reason for this is that, as is standard in submanifold geometry, 
the codimension one case is the first step towards a more general investigation, 
and at the same time it usually provides interesting geometric phenomena and examples. 
However, it is important to mention at this point that we cannot expect to gather in our classifications some of the examples of Einstein submanifolds mentioned above, 
except for some trivial cases. 
Indeed, it is well-known that the only irreducible symmetric spaces which admit totally geodesic hypersurfaces are those of constant curvature and, 
by definition, none of the examples in \cite{Tamaru} is of codimension one either. 

Our first result is the classification of codimension one subgroups of the solvable Iwasawa group $AN$ of an irreducible symmetric space of non-compact type $M$ 
which are Ricci solitons.

\begin{theo}\label{th:symmetric}
	Let $S$ be a codimension one connected Lie subgroup of the solvable Iwasawa group $A N$
of an irreducible symmetric space of non-compact type $M$. Then $S$ is a Ricci soliton 
with respect to the metric induced by the left-invariant Einstein metric on $AN$ if and only~if:
	\begin{enumerate}[\rm(i)]
		\item $M$ is arbitrary and $S$ contains the nilpotent part $N$,
		\item $M$ is a complex hyperbolic plane $\C \mathsf{H}^2$ and $S$ is a Lohnherr hypersurface $W^3$, or 
		\item $M$ is a real hyperbolic space $\R\mathsf{H}^n$ and $S$ is arbitrary. 
	\end{enumerate}
\end{theo}

In item~(i), 
$S$ is a codimension one connected Lie subgroup of $AN$ containing $N$ if and only if its Lie algebra is of the form 
$(\g{a}\ominus\ell)\oplus\g{n}$, 
where $\g{a}\oplus\g{n}$ is the Lie algebra of $A N$, 
$\ominus$ denotes orthogonal complement, and $\ell$ is a one-dimensional vector subspace of $\g{a}$. 
In the rank one case, such an $S$ is precisely $N$. 
In the higher rank case, there are continuously many isometry classes of such subgroups $S$~\cite{BT:crelle}, 
and some of them are Einstein~\cite[\S3]{CHKTT18}. 
We note that the class of examples in item~(i) contains the horospheres of $M$, 
that is, the level sets of Busemann functions~\cite[\S1.10]{eberlein} (for details see Remark~\ref{rem:horospheres}).
We also recall that a Lohnherr hypersurface $W^{2n-1}$ (also called a fan) is the unique (up to congruence) minimal homogeneous hypersurface of a complex hyperbolic space $\C \mathsf{H}^n$~\cite{BD:gd}. It can be defined as~the connected Lie subgroup of $AN\cong\C\mathsf{H}^n$ with Lie algebra $\g{a}\oplus(\g{n}\ominus\ell)$, where $\ell$ is a one-dimensional subspace of the only simple root space $\g{g}_\alpha\subset\g{n}=\g{g}_\alpha\oplus\g{g}_{2\alpha}$. Finally, in relation to item (iii) of Theorem~\ref{th:symmetric}, we comment that any codimension one subgroup of~the Iwasawa group of a real hyperbolic space is a space form of constant curvature, and~hence~Einstein.

We will sketch below in this introduction the common approach to the main results of this paper, including Theorem~\ref{th:symmetric}. 
But at this point we would like to draw the attention to the main difficulties for proving Theorem~\ref{th:symmetric}. 
One is that codimension one subgroups $S$ of $A N$ form a rich class. 
There is a one-parameter family of isometry classes of such $S$ for the rank one case, and even a larger family for the higher rank case, depending on the root system. 
We need to study the Ricci soliton condition for all of them. 
Another difficulty is the very determination of the Ricci tensor of the hypersurface $S$ of $M$. 
The reason is that the Levi-Civita connection of a symmetric space of non-compact type, 
despite admitting a Lie algebraic description, turns out to be quite hard to handle in full generality, 
since it involves Lie brackets which relate (positive) root spaces in a complicated way. 
One sign of this difficulty is that, previously to this article, a classification as in Theorem~\ref{th:symmetric} had only been obtained, apart from the well-known case of constant curvature $M\cong\R\mathsf{H}^n$, 
in the very specific setting of the complex hyperbolic space $M\cong\C \mathsf{H}^n$~\cite{HKT:tohoku}. 
Note that both cases are of rank one, and hence the associated root systems only have one simple root. 
There are some other studies in particular rank two spaces~\cite{Suh}, but the hypersurfaces are assumed to satisfy some additional conditions. 
In this paper, we address and handle the problem for a general root system, yielding the (as far as we know) first general and systematic investigation of Ricci soliton hypersurfaces in the whole family of symmetric spaces of non-compact type. 

We would also like to note that Theorem~\ref{th:symmetric} could be stated in a more general way by assuming that $S$ is a Lie hypersurface of $M$. Here, by Lie hypersurface we mean a codimension one orbit of a cohomogeneity one action on $M$ with no singular orbits. Lie hypersurfaces have been classified by Berndt and Tamaru~\cite{BT:jdg}, and they turn out to be congruent, precisely, to codimension one Lie subgroups of the solvable Iwasawa groups. However, unlike other ambient manifolds (such as Euclidean spaces, spheres or, more generally, irreducible symmetric spaces of compact type), the classification of cohomogeneity one actions (or, equivalently, homogeneous hypersurfaces) on symmetric spaces of non-compact type is still an outstanding open problem. Although there are partial classification results (see~\cite{BT:crelle}), and in particular Lie hypersurfaces have been classified~\cite{BT:jdg}, there is a remarkable richness of examples, and analyzing their geometry is usually a difficult problem; see~\cite{DDR:hhn} for a recent contribution. 

The proof of Theorem~\ref{th:symmetric} relies on the study of two different cases: rank one and higher rank. Whereas, as already mentioned, the higher rank cases require a careful analysis of the geometry of Lie hypersurfaces in terms of restricted root systems, the case of rank one symmetric spaces is subsumed into the investigation of a broader family of Einstein solvmanifolds, namely, Damek-Ricci harmonic spaces~\cite{BTV}. These are well-known solvable extensions of the so-called generalized Heisenberg groups (or $H$-type groups), which in turn constitute an important family of two-step nilpotent metric Lie groups. Roughly speaking (see~\S\ref{subsec:Heisenberg_prelim} for details), a generalized Heisenberg group is a simply connected Lie group $N$ with Lie algebra $\g{n}=\g{v}\oplus\g{z}$, where $\g{z}$ is the center of $\g{n}$, $\g{v}$ is a Clifford module over $\g{z}$, and $N$ is endowed with certain natural left-invariant metric. A Damek-Ricci space is a simply connected semidirect product  $AN$ with left-invariant metric, where $N$ is a generalized Heisenberg group and $A$ is a one-dimensional Lie group (see~\S\ref{subsec:DR_prelim}). 
The solvable Iwasawa groups of rank one symmetric spaces 
(apart from real hyperbolic spaces $\R \mathsf{H}^n$)
are, indeed, the only symmetric Damek-Ricci spaces. 

In line with our interest in codimension one in this paper, we will first derive the classification of codimension one Ricci soliton subgroups of generalized Heisenberg groups. We denote a generalized Heisenberg group by $N(m, k)$ or $N(m, k_+,k_-)$, 
where $m$ is the dimension of the center, 
and $k$ or $(k_+,k_-)$ represents the number of irreducible factors of the corresponding Clifford modules; we refer to \S\ref{subsec:Heisenberg_prelim} for precise definitions.
\begin{theo}\label{th:Heisenberg}
Let $S$ be a codimension one connected Lie subgroup of a generalized Heisenberg group $N$. Let $\g{s} = \g{n} \ominus \R \xi$ be the corresponding Lie subalgebra of $\g{n}$. Then $\xi \in \g{v}$, and $S$ is a Ricci soliton with respect to the induced metric if and only if $N$ is isometric to:
	\begin{enumerate}[\rm(i)]
		\item  $N(1,k)$ for some $k\geq 1$,
		\item  $N(2,1)$, 
		\item  $N(m,1,0)$, where $m\in\{3,7\}$, or
		\item  $N(m,1)$, where $m\in\{4,8\}$, and $\xi$ belongs to one of the two $m$-dimensional irreducible half-spin submodules of the $\Spin(m)$-module $\g{v}$.
	\end{enumerate}
\end{theo}

Some of the Ricci soliton nilmanifolds $S$ obtained above would be interesting also from an intrinsic viewpoint. 
The most complicated one is the codimension one subgroup $S$ in $N(8,1)$. 
In this case $S$ is a two-step nilpotent $23$-dimensional Lie group with $8$-dimensional center, 
which would have not been considered in the literature as far as the authors know.

The analysis developed in the proof of Theorem~\ref{th:Heisenberg} will then be useful to obtain the classification in Damek-Ricci spaces.

\begin{theo}\label{th:DR}
	Let $S$ be a codimension one connected Lie subgroup of a Damek-Ricci space $AN$. Then $S$ is a Ricci soliton with respect to the induced metric if and only if $S=N$ is the generalized Heisenberg group associated with $AN$, or $AN$ is isometric to a complex hyperbolic plane $\C \mathsf{H}^2$ and $S$ is the Lohnherr hypersurface $W^3$.
\end{theo}

The proofs of Theorems~\ref{th:symmetric}, \ref{th:Heisenberg} and~\ref{th:DR} rely on the same simple underlying idea: use the Gauss equation of submanifold geometry to calculate the Ricci operator of an arbitrary codimension one subgroup, and then decide when the induced metric of such subgroup is an algebraic Ricci soliton. We recall that a left-invariant metric $g$ on a Lie group $S$ is an algebraic Ricci soliton if there exist a derivation $D\in\Der(\g{s})$ and a real number $c$ such~that
\[
\Ric=c \id + D,
\] 
where $\Ric$ is the $(1,1)$-Ricci tensor of $(S,g)$. 
Algebraic Ricci solitons on simply connected groups are Ricci solitons and, indeed, all known examples of expanding homogeneous Ricci solitons are isometric to simply connected solvsolitons; we recall that a solvsoliton (resp.\ nilsoliton) is, precisely, a solvable (resp.\ nilpotent) Lie group $S$ with an algebraic Ricci soliton metric. Various fundamental results have been proved over the last few years regarding the relation between homogeneous Ricci solitons and algebraic Ricci solitons (see for example~\cite{Jab:gt, Jab:crelle, LafuenteLauret:jdg, Lauret:mathann, Lauret:crelle}). In particular, it is known that any left-invariant Ricci soliton metric on a completely solvable (or in particular,  nilpotent) Lie group is necessarily a solvsoliton.  
Since all codimension one subgroups $S$ studied in this paper are completely solvable, 
we can reduce our problem of determining when $S$
is a Ricci soliton (with the induced left-invariant metric) to the analysis of the algebraic Ricci soliton condition on $S$. 

As mentioned above, the first basic step in our arguments is to calculate the Ricci operator of each codimension one subgroup $S$. This is done via the Gauss equation in terms of the shape operator and the normal Jacobi operator of $S$. Whereas for generalized Heisenberg groups and Damek-Ricci spaces (Theorems~\ref{th:Heisenberg} and~\ref{th:DR}) this is a relatively straightforward task, dealing with symmetric spaces (Theorem~\ref{th:symmetric}) is much more involved, as this requires a meticulous inspection of the extrinsic geometry of $S$ in terms of the root space decomposition of the isometry Lie algebra of the symmetric space.
Once we have calculated the Ricci operator $\Ric$ of $S$, the classifications are reduced to the determination of the conditions for which the operator $\Ric -c\id$ is a derivation of $\g{s}$, for some $c\in\R$. This is a purely algebraic but non-trivial problem which requires different approaches depending on the ambient space where $S$ lives. Thus, for generalized Heisenberg groups and Damek-Ricci spaces we develop an analysis involving Clifford modules and spin representations, whereas for symmetric spaces we are led again to arguments involving root spaces.


This article is structured as follows. In Section~\ref{sec:prelim} we basically fix some notation and recall some well-known facts concerning Riemannian submanifold geometry. Section~\ref{sec:Heisenberg} is devoted to the study of Ricci soliton hypersurfaces in generalized Heisenberg groups and the proof of Theorem~\ref{th:Heisenberg}. Section~\ref{sec:DR} deals with Damek-Ricci spaces and contains the proof of Theorem~\ref{th:DR}. Finally, in Section~\ref{sec:symmetric} we investigate Ricci soliton Lie hypersurfaces in symmetric spaces of non-compact type, and derive the proof of Theorem~\ref{th:symmetric}. Each one of these three main sections contains a subsection with preliminaries on generalized Heisenberg groups, Damek-Ricci spaces and symmetric spaces of non-compact type, respectively.

\medskip

\textbf{Acknowledgments.} The authors would like to thank Eduardo Garc\'ia-R\'io and Alberto Rodr\'iguez-V\'azquez for some helpful comments.

\section{Preliminaries}\label{sec:prelim}
In this short section we introduce some basic notation and facts concerning Riemannian geometry of hypersurfaces.
 
Let $M$ be a Riemannian manifold with metric $\langle\cdot,\cdot\rangle$ and Levi-Civita connection $\nabla$. We will adopt the sign convention $R^M(X,Y)Z=\nabla_X\nabla_Y Z-\nabla_Y\nabla_X Z-\nabla_{[X,Y]}Z$ for the definition of the curvature tensor $R^M$ of $M$, where $X$, $Y$, $Z$ are smooth vector fields on $M$. We will denote the Ricci operator of $M$ by $\Ric^M$.

Now, let $S$ be a hypersurface of the Riemannian manifold $M$. Let $\xi$ be a smooth unit normal vector field on (an open subset of) $S$. The shape operator of $S$ is the endomorphism $\Ss_\xi$ of $TS$ given by $\Ss_\xi X:=-(\nabla_X\xi)^\top=-\nabla_X\xi$, for each $X\in TS$, where ${}^\top$ denotes orthogonal projection onto the tangent space of $S$. The shape operator is a self-adjoint endomorphism of the tangent space of $S$. The principal curvatures of $S$ are precisely the (real) eigenvalues of $\Ss_\xi$, and the mean curvature of $S$ is the trace $\tr\Ss_\xi$. Another important object which relates the geometry of the ambient space $M$ with that of the hypersurface $S$ is the Jacobi operator of $S$, defined as the endomorphism $R_\xi$ of $TS$ given by $R_\xi(X):=R^M(X,\xi)\xi$, for each $X\in TS$. Again, the Jacobi operator is self-adjoint.

The extrinsic geometry of a submanifold is controlled by the fundamental equations of submanifold geometry. One of these relations is Gauss equation, which, for a hypersurface $S$ of $M$ as above, can be written as
\[
\langle R^M(X,Y)Z,W\rangle=\langle R(X,Y)Z,W\rangle-\langle \Ss_\xi Y,Z\rangle\langle \Ss_\xi X,W\rangle
+\langle\Ss_\xi X,Z\rangle\langle \Ss_\xi Y,W\rangle,
\]
for tangent vectors $X$, $Y$, $Z$, $W$ to $S$, and where $R$ is the curvature tensor of $S$. From this equation, one can easily derive the following expression for the Ricci operator $\Ric$ of $S$ in terms of the Ricci operator of $M$ and the shape and Jacobi operators of $S$:
\begin{equation}\label{eq:Ric_Gauss}
\Ric =(\Ric^M\rvert_{TS})^\top+\tr(\Ss_\xi) \Ss_\xi-\Ss_\xi^2-R_\xi.
\end{equation}

In this article we will work with some particular types of homogeneous hypersurfaces in different ambient spaces. Recall that by a homogeneous hypersurface we understand a codimension one orbit $S$ of an isometric action of a (connected) Lie group on an ambient space $M$. Homogeneous hypersurfaces have constant principal curvatures with constant multiplicities, and the same happens with the eigenvalues of the Jacobi operator. In this paper, $M$ will always be isometric to a Lie group with a left-invariant metric, and $S$ a codimension one Lie subgroup of $M$. In this setting, the tangent space of $S$ at each point is spanned by the left-invariant vector fields in the Lie algebra $\g{s}$ of $S$, and hence $\xi$ is a left-invariant unit normal vector field globally defined on $S$.

\section{Nilsoliton hypersurfaces in generalized Heisenberg groups}\label{sec:Heisenberg}
In this section we determine which codimension one Lie subgroups of generalized Heisenberg groups are Ricci solitons with the induced metric. In \S\ref{subsec:Heisenberg_prelim} we recall from~\cite{BTV} the definition and basic facts about generalized Heisenberg groups, whereas in \S\ref{subsec:Heisenberg_proof} we prove the classification result contained in Theorem~\ref{th:Heisenberg}.

\subsection{Generalized Heisenberg groups}\label{subsec:Heisenberg_prelim}\hfill

Let $\g{n}=\g{v}\oplus\g{z}$ be a Lie algebra equipped with a positive definite inner product $\langle \cdot,\cdot \rangle$ such that  $\langle \g{v},\g{z}\rangle=0$, and whose Lie bracket satisfies $[\g{v},\g{v}]\subset\g{z}$ and $[\g{n},\g{z}]=0$. In this section we will assume that $\g{v}\neq 0\neq\g{z}$, since otherwise $\g{n}$ would be an abelian Lie algebra. Define a linear map $J\colon \g{z}\to \mathrm{End}(\g{v})$ by
\[
\langle J_Z U, V\rangle=\langle [U,V],Z\rangle, \qquad \text{for all $U$, $V\in\g{v}$, $Z\in\g{z}$.}
\]
Then, the two-step nilpotent Lie algebra $\g{n}$ is called a \emph{generalized Heisenberg algebra}  or \emph{H-type algebra} if
\[
J_Z^2=-\langle Z,Z\rangle \id, \qquad \text{for all $Z\in\g{z}$.}
\]
The simply connected nilpotent Lie group $N$ with Lie algebra $\g{n}$, equipped with the left-invariant metric induced by $\langle \cdot,\cdot \rangle$, is called a \emph{generalized Heisenberg group} or \emph{H-type group}. 

The map $J$ induces a representation of the Clifford algebra $\Cl(m)=\Cl(\g{z},q)$ on $\g{v}$, where $m=\dim\g{z}$ and $q=-\langle\cdot,\cdot\rangle\rvert_{\g{z}\times\g{z}}$ is a negative definite quadratic form. 
Conversely, a representation of such a Clifford algebra induces a map $J$ as above and, hence, a generalized Heisenberg group. 
Therefore, the classification of these metric Lie groups follows from the classification of Clifford modules. If $m\not\equiv 3 \mod{ 4}$, then there exists exactly one irreducible Clifford module $\g{d}$ over $\Cl(m)$, up to equivalence, and each Clifford module over $\Cl(m)$ is isomorphic to $\g{v}\cong\oplus^k \g{d}$. We will denote the corresponding generalized Heisenberg group by $N(m, k)$. If $m\equiv 3 \mod 4$, then there are exactly two irreducible Clifford modules, $\g{d}_+$ and $\g{d}_-$, over $\Cl(m)$, up to equivalence; they satisfy $\dim\g{d}_+=\dim\g{d}_-$, and each Clifford module over $\Cl(m)$ is isomorphic to $\g{v}\cong\bigl(\oplus^{k_+}\g{d}_+\bigr)\oplus\bigl(\oplus^{k_-}\g{d}_-)$. We will denote the corresponding group by $N(m,k_+,k_-)$. In any case, we will denote by $k$ the number of irreducible $\Cl(m)$-submodules in $\g{v}$; thus, $k=k_++k_-$ if $m\equiv 3 \mod 4$. In particular, if $n=\dim \g{v}$, we have  $n=k\dim\g{d}_{(\pm)}$.
As vector spaces, the irreducible modules $\g{d}$ or $\g{d}_{\pm}$ are isomorphic to the vector spaces shown in Table~\ref{table:delta}, where the explicit dimensions of $\g{v}$ are also included. Moreover, if $m\equiv 3 \mod 4$, then $N(m,k_+,k_-)$ is isometric to $N(m,k_+',k_-')$ if and only if $(k_+',k_-')\in\{(k_+,k_-),(k_-,k_+)\}$.

\begin{table}[h]
	\renewcommand{\arraystretch}{1.5}
	\begin{tabular}{ccccccccc}
		\hline
		$m$ & $8p$ & $8p+1$ & $8p+2$ & $8p+3$ & $8p+4$ & $8p+5$ & $8p+6$ & $8p+7$
		\\ \hline
		$\g{d}_{(\pm)}$ & $\R^{2^{4p}}$ & $\C^{2^{4p}}$ & $\H^{2^{4p}}$ & $\H^{2^{4p}}$
		& $\H^{2^{4p+1}}$ & $\C^{2^{4p+2}}$ & $\R^{2^{4p+3}}$ & $\R^{2^{4p+3}}$
		\\
		\hline
		$n$ & $2^{4p}k$ & $2^{4p+1}k$ & $2^{4p+2}k$ & $2^{4p+2}k$ & $2^{4p+3}k$ & $2^{4p+3}k$ & $2^{4p+3}k$ & $2^{4p+3}k$ 
		\\
		\hline
	\end{tabular}
	\bigskip
	\caption{Clifford modules.}\label{table:delta}
	\vspace{-2ex}
\end{table}

Let $U,V\in\g{v}$ and $X,Y\in\g{z}$. Then the following relations hold:
\begin{equation}\label{eq:inner_gHg}
\langle J_X U,J_X V\rangle =\lvert X\rvert^2 \langle U,V\rangle,\qquad \langle J_X U,J_Y U\rangle =\langle X,Y\rangle \lvert U\rvert^2, \qquad [U,J_X U]=\vert U\vert^2 X.
\end{equation} 
In particular, the first and third equalities imply that $J_X$ is an orthogonal complex structure on $\g{v}$ for any unit $X\in\g{z}$, and $[\g{v}, \g{v}] = \g{z}$. Finally, we recall the formula of the Levi-Civita connection of a generalized Heisenberg group,
\begin{equation}\label{eq:LC_H}
\nabla_{V+Y}(U+X)=-\frac{1}{2}J_XV-\frac{1}{2}J_YU-\frac{1}{2}[U,V],
\end{equation}
which will be fundamental in our work, as well as 
the Ricci operator $\Ric^N$ of a generalized Heisenberg group $N$,
\begin{equation}\label{eq:ricci_gHg}
\Ric^N\vert_\g{v}=-\frac{m}{2}\id, \qquad \Ric^N\vert_\g{z}=\frac{n}{4}\id.
\end{equation}
We refer to~\cite[Chapter~3]{BTV} for more details on these facts. We conclude this subsection by observing that any generalized Heisenberg group is an algebraic Ricci soliton. In fact, it is straightforward that the endomorphism $D$ of $\g{n}=\g{v}\oplus\g{z}$ determined by $D\rvert_\g{v}=\id$ and $D\rvert_\g{z}=2\id$ is a derivation of $\g{n}$, and $\Ric^N=\wt{D}+c\id$ with $\wt{D}=(n/4+m/2)D$ and $c=-m-n/4$.

\subsection{Proof of the classification result}\label{subsec:Heisenberg_proof}\hfill

Let $S$ be a connected Lie subgroup of codimension one of a generalized Heisenberg group~$N$. Let $\g{s}=\g{n}\ominus\R \xi$ be the corresponding Lie subalgebra of  $\g{n}=\g{v}\oplus\g{z}$, where $\ominus$ denotes (here and henceforth) orthogonal complement, and $\xi=W+Z\in\g{n}$ is a unit vector with $W\in \g{v}$, $Z\in\g{z}$. Then $J_Z W$, $|Z|^2W-|W|^2Z\in\g{s}$, and hence by~\eqref{eq:inner_gHg}
\begin{equation}\label{eq:xi_H_DR}
[J_Z W,|Z|^2W-|W|^2Z]=-|Z|^2 |W|^2 Z\in\g{s}.
\end{equation}
This, together with $[\g{v},\g{v}]=\g{z}$, implies that $\xi\in\g{v}$. Conversely, if $\xi \in \g{v}$, then one can easily see that 
$\g{s}=\g{n} \ominus \R \xi = (\g{v} \ominus \R \xi) \oplus \g{z}$ is a Lie subalgebra of $\g{n}$.

In what follows, we will use the following notation:
\[
\g{J}=\{J_Z:Z\in\g{z}\},\qquad \g{J}\xi=\{J_Z\xi:Z\in\g{z}\},\qquad \text{and}\qquad (\g{J}\xi)^\perp=\g{v}\ominus(\R\xi\oplus\g{J}\xi).
\]
Note that we have the orthogonal direct sum decomposition $\g{s}=\g{J}\xi\oplus(\g{J}\xi)^\perp\oplus\g{z}$, where $\dim\g{J}\xi=\dim\g{z}=m$ and $\dim(\g{J}\xi)^\perp=n-m-1$.

We start by calculating the shape operator $\Ss_\xi$ of  the hypersurface $S$ of $N$ by using \eqref{eq:LC_H} along with the relations~\eqref{eq:inner_gHg}:
\begin{align*}
\Ss_\xi U&=-\nabla_U\xi=\frac{1}{2}[\xi,U]=0, &\qquad &\text{for any } U\in(\g{J}\xi)^\perp,
\\
\Ss_\xi J_Z\xi &= -\nabla_{J_Z\xi}\xi=\frac{1}{2}[\xi,J_Z\xi]=\frac{1}{2}Z,& \qquad &\text{for any }Z\in\g{z},\, J_Z\xi\in \g{J}\xi, 
\\
\Ss_\xi Z&=-\nabla_Z \xi=\frac{1}{2}J_Z\xi, &\qquad &\text{for any } Z\in\g{z}.
\end{align*}
Hence, both $(\g{J}\xi)^\perp$ and $\g{J}\xi\oplus\g{z}$ are invariant under $\Ss_\xi$ and, moreover,
\begin{equation}\label{eq:shape_gHg}
\tr (\Ss_\xi)=0, \qquad \Ss_\xi^2\vert_{(\g{J}\xi)^\perp}=0, \qquad  \Ss_\xi^2\vert_{\g{J}\xi\oplus\g{z}}=\frac{1}{4}\id.
\end{equation}

Now, we calculate the normal Jacobi operator $R_\xi=R^N(\cdot,\xi)\xi$ of $S$ by combining the definition of the curvature tensor of the ambient space, $R^N(X,Y)=[\nabla_X,\nabla_Y]-\nabla_{[X,Y]}$, with the formula~\eqref{eq:LC_H} for the Levi-Civita connection, or alternatively by using the formula for the Jacobi operator of a generalized Heisenberg group~\cite[\S3.1.8]{BTV}. In any case, we get
\begin{equation}\label{eq:Jacobi_gHg}
R_\xi U=0, \qquad R_\xi J_Z\xi=-\frac{3}{4}J_Z\xi, \qquad R_\xi Z=\frac{1}{4}Z,
\end{equation}
for every $U\in (\g{J}\xi)^\perp$ and $Z\in \g{z}$.

Inserting \eqref{eq:ricci_gHg}, \eqref{eq:shape_gHg} and~\eqref{eq:Jacobi_gHg} into~\eqref{eq:Ric_Gauss}, we get that $(\g{J}\xi)^\perp$, $\g{J}\xi$ and $\g{z}$ are invariant subspaces for the Ricci operator $\Ric$ of the hypersurface $S$. Moreover $\Ric$ is given by
\begin{equation}\label{eq:ricci'}
\Ric\rvert_{(\g{J}\xi)^\perp}=-\frac{m}{2}\id, \qquad \Ric\rvert_{\g{J}\xi}=\frac{1-m}{2}\id,\qquad
\Ric\rvert_{\g{z}}=\frac{n-2}{4}\id.
\end{equation}

At this point, we will organize the arguments towards the proof of Theorem~\ref{th:Heisenberg} into several propositions that will allow us to analyze in which cases $S$ is an algebraic Ricci soliton. Thus, we start with the following result.

\begin{proposition}\label{prop:Heisenberg_charact}
	$S$ is an algebraic Ricci soliton if and only if at least two of the following three conditions hold:  $\g{J}\xi$ is abelian, $(\g{J}\xi)^\perp$ is abelian, and $[\g{J}\xi,(\g{J}\xi)^\perp]=0$.
\end{proposition}
\begin{proof}
	First note that $S$ is an algebraic Ricci soliton if and only if there exists $c\in\R$ such that the endomorphism $D:=\Ric -c \id$ of $\g{s}$, which in view of~\eqref{eq:ricci'} is given by 
	\begin{equation}
	D\rvert_{(\g{J}\xi)^\perp}=-\left(\frac{m}{2}+c\right)\id, \qquad D\rvert_{\g{J}\xi}=\left(\frac{1-m}{2}-c\right)\id,\qquad
	D\rvert_{\g{z}}=\left(\frac{n-2}{4}-c\right)\id,
	\end{equation}
	is a derivation of $\g{s}$, that is, for any $U,V\in\g{s}$ it satisfies
	\begin{equation}\label{eq:derivation}
	D[U,V]=[DU,V]+[U,DV].
	\end{equation}
	
	Given any $U$, $V\in\g{J}\xi$, we have
	\[D[U,V]=\left(\frac{n-2}{4}-c\right)[U,V],\qquad [DU,V]=[U,DV]=\left(\frac{1-m}{2}-c\right)[U,V].
	\]
	Then \eqref{eq:derivation} holds for every $U,V\in\g{J}\xi$ if and only if $c=(6-4m-n)/4$ or $\g{J}\xi$ is abelian. Arguing similarly, we deduce that \eqref{eq:derivation} holds for every $U,V\in(\g{J}\xi)^\perp$ if and only if $c=(2-4m-n)/4$ or $(\g{J}\xi)^\perp$ is abelian; and \eqref{eq:derivation} holds for every $U\in\g{J}\xi$ and $V\in(\g{J}\xi)^\perp$ if and only if $c=(4-4m-n)/4$ or $[\g{J}\xi,(\g{J}\xi)^\perp]= 0$. 
	
	Therefore, if $S$ is an algebraic Ricci soliton (i.e.\ $D\in\Der(\g{s})$), then at least~two~of the conditions in the statement hold (since otherwise we would get two different values for~$c$).
	
	Conversely, since $\g{z}$ is the center of $\g{n}$, and $D$ leaves both $\g{v}\ominus\R\xi$ and $\g{z}$ invariant, it turns out that, for a fixed $c\in\R$, $D\in\Der(\g{s})$ if and only if \eqref{eq:derivation} holds for all $U$, $V\in\g{J}\xi\oplus(\g{J}\xi)^\perp$. But, in view of the equivalences in the previous paragraph, if at least two of the conditions in the statement are satisfied, then there exists $c\in\R$ satisfying \eqref{eq:derivation} for all $U$, $V\in\g{J}\xi\oplus(\g{J}\xi)^\perp$.
\end{proof}

Note that, from the definition of $J\colon\g{z}\to\mathrm{End}(\g{v})$, we deduce that, given $U$, $V\in\g{v}$,~we~have $[U,V]=0$ if and only if $\langle J_Z U,V\rangle=0$ for all $Z\in\g{z}$. Hence, a subspace $\g{w}$ of $\g{v}$ is abelian if and only if $\g{J}\g{w}\perp \g{w}$. It will be helpful to take this remark into account in what follows.

\begin{proposition}\label{prop:perp_abelian}
	$(\g{J}\xi)^\perp$ is abelian if and only if one of the following conditions holds:
	\begin{enumerate}[\rm(i)]
		\item $\dim (\g{J}\xi)^\perp\in\{0,1\}$, 
		\item $N$ is isomorphic to $N(m,1)$, where $m\in\{4,8\}$, and $\xi$ belongs to one of the two $m$-dimensional irreducible half-spin submodules of the $\Spin(m)$-module $\g{v}$.
	\end{enumerate}
Moreover, in case \emph{(ii)}, $\g{J}\xi$ is abelian.
\end{proposition}
\begin{proof}
	Assume that  $(\g{J}\xi)^\perp$ is abelian. Then, if $U\in (\g{J}\xi)^\perp$ we have $\langle J_Z U, V\rangle=\langle [U,V],Z\rangle=0$ for each $V\in (\g{J}\xi)^\perp$ and $Z\in\g{z}$. Moreover, $\langle J_Z U, \xi\rangle=-\langle J_Z \xi, U\rangle=0$. Hence, $\g{J}U\subset\g{J}\xi$, for each $U\in (\g{J}\xi)^\perp$. In particular, for any fixed non-zero $Z\in\g{z}$, we have that $J_Z$ maps $(\g{J}\xi)^\perp$ injectively into $\g{J}\xi$. Therefore, $n-m-1=\dim (\g{J}\xi)^\perp\leq \dim\g{J}\xi=m$, which implies that $n\leq 2m+1$. But according to the classification of generalized Heisenberg groups and their dimensions (or, equivalently, according to the classification of Clifford modules over $\Cl(m)$, see~Table~\ref{table:delta}), there is only a finite number of groups satisfying such inequality, namely the groups with $1\leq m\leq 8$ and $k=1$. For $m\in\{1,3,7\}$ we have $n=m+1$ and hence $\dim (\g{J}\xi)^\perp=0$, whereas for $m\in\{2,6\}$ we have $n=m+2$ and then $\dim (\g{J}\xi)^\perp=1$. We are left with the cases $m\in\{4,5,8\}$.
	
	For $N(5,1)$ we have $n=8$, whence $\dim (\g{J}\xi)^\perp=2$. If we restrict the Clifford representation of $\Cl(5)$ on $\g{v}$ to a $\Cl(3)$-module, then $\g{v}$ turns out to be isomorphic to $\H\oplus\H$. Since the action of $\Spin(5)\cong\mathsf{Sp}(2)$ on $\g{v}$ is transitive on the unit sphere of $\g{v}$, we may assume (by conjugating the Clifford subalgebra $\Cl(3)$ by an element of $\Spin(5)$ if necessary) that $\xi$ is contained in one of these $\H$-factors, and then $(\g{J}\xi)^\perp$ is contained in the other $\H$-factor. In other words, $\g{v}=\R\xi\oplus\widetilde{\g{J}}\xi\oplus \R U\oplus \widetilde{\g{J}}U$, where we take $U\in(\g{J}\xi)^\perp$ and $\widetilde{\g{J}}=\spann\{J_{Z_1},J_{Z_2},J_{Z_3}\}$ for some orthonormal subset $Z_1,Z_2,Z_3$ of $\g{z}$. Then there exists $J\in\widetilde{\g{J}}$ such that $(\g{J}\xi)^\perp=\R U\oplus\R JU$, from where we deduce that $(\g{J}\xi)^\perp$ is not abelian, which contradicts our initial assumption.
	
	In order to finish the proof of the necessity of the proposition we have to analyze~the~cases $N(m,1)$, $m\in\{4,8\}$. Note that in both cases $n=\dim\g{v}=2m$, whence $\dim (\g{J}\xi)^\perp=m-1$.
	
	Assume that $N=N(m,1)$, $m\in\{4,8\}$. Since $(\g{J}\xi)^\perp$ is abelian by assumption, then for each non-zero $J\in\g{J}$ we have $J(\g{J}\xi)^\perp\perp(\g{J}\xi)^\perp$, that is, $J(\g{J}\xi)^\perp \subset \R \xi \oplus \g{J} \xi$. 
	Here one can easily see that $\xi$ and $J \xi$ are perpendicular to $J(\g{J}\xi)^\perp$, 
	and hence $J(\g{J}\xi)^\perp\subset \g{J}\xi\ominus\R J\xi$; but since both subspaces have the same dimension $m-1$, we have $J(\g{J}\xi)^\perp= \g{J}\xi\ominus\R J\xi$. Then also $J(\g{J}\xi\ominus\R J\xi)=(\g{J}\xi)^\perp$, for any non-zero $J\in\g{J}$.
	Take an orthonormal basis $\{Z_1,\dots, Z_m\}$ of $\g{z}$ and put $J_i=J_{Z_i}$, $i\in\{1,\dots,m\}$. Then $\{J_iJ_j\xi:1\leq i<j\leq m\}$ spans $(\g{J}\xi)^\perp$. 
	Consider the $\Spin(m)$-action on $\g{v}$ induced by the $\Cl(m)$-module $\g{v}$. On the one hand, this action is known to be the sum of the two irreducible half-spin representations, $\g{v}=\Delta_+\oplus\Delta_-$ (see for example~\cite[Proposition~I.5.12]{LM}). This representation restricts to the unit sphere $\mathbb{S}^{2m-1}$ of $\g{v}$, and the resulting action is of cohomogeneity one (see~\cite[Table~1]{Bergmann} for $m=8$, whereas if $m=4$ then $\Spin(4)\cong\mathsf{Sp}(1)\times\mathsf{Sp}(1)$ acts factorwise on $\g{v}\cong\H\oplus\H$). In fact, 
	the orbit $\Spin(m) \cdot \xi$ is singular if and only if $\xi$ is in one of the irreducible components $\Delta^+$ and $\Delta^-$. 
	In this case the orbit is isometric to the sphere $\mathbb{S}^{m-1}$, 
	to be exact, 
	$(\mathsf{Sp}(1) \times \mathsf{Sp}(1))/\mathsf{Sp}(1) \cong \mathbb{S}^3$ or $\Spin(8)/\Spin(7) \cong \mathbb{S}^7$. 
	Note that the regular orbits are 
	$\mathsf{Sp}(1) \times \mathsf{Sp}(1) \cong \mathbb{S}^3 \times \mathbb{S}^3$ and $\Spin(8)/\mathsf{G}_2 \cong \mathbb{S}^7 \times \mathbb{S}^7$, 
	which are of codimension one in the unit sphere of $\g{v}$. On the other hand, we can compute the tangent space to the orbit through $\xi$ as
	\[
	T_\xi(\Spin(m)\cdot \xi)=\g{spin}(m)\cdot \xi=\spann\{J_iJ_j\xi:1\leq i<j\leq m\}=(\g{J}\xi)^\perp.
	\]
	Hence, $\Spin(m)\cdot \xi$ is $(m-1)$-dimensional and, therefore, it must be one of the singular orbits, which implies that $\xi\in\Delta_+$ or $\xi\in\Delta_-$. This completes one of the implications in the statement.	
	
	Let us now show the converse. If  $\dim (\g{J}\xi)^\perp\in\{0,1\}$, then $(\g{J}\xi)^\perp$ is trivially abelian. So let us assume that  $N$ is isomorphic to $N(m,1)$ with $m\in\{4,8\}$. The decomposition $\g{v}=\Delta_+\oplus\Delta_-$ coincides with the $(\pm 1)$-eigenspace decomposition of the volume element $\omega=J_1\cdots J_m$ (see~\cite[Proposition~I.5.10]{LM}), and then any non-zero $J\in\g{J}$ interchanges $\Delta_+$ and $\Delta_-$ (see~\cite[Proposition~I.3.6]{LM}). Hence, since $\g{v}=\Delta_+\oplus\Delta_-$ is an orthogonal direct sum, we get that $\Delta_+$ and $\Delta_-$ are abelian. Now, if $\xi\in\Delta_+$ then $\g{J}\xi=\Delta_-$. This proves the last claim of the statement, but it also implies that  $(\g{J}\xi)^\perp=\Delta_+\ominus\R\xi$, and hence $(\g{J}\xi)^\perp$ is abelian. We argue analogously if $\xi\in\Delta_-$, and this concludes the proof.
\end{proof}

\begin{remark}\label{rem:dimJxi01}
	Notice that $\dim (\g{J}\xi)^\perp=0$ if and only if $N$ is isomorphic to $N(1,1)$, $N(3,1,0)$ or $N(7,1,0)$. Moreover $\dim (\g{J}\xi)^\perp=1$ corresponds precisely to $N(2,1)$ and $N(6,1)$. If $N$ is isomorphic to $N(2,1)$ then $\g{J}\xi$ is abelian. Indeed, in this case, if $\{Z_1,Z_2\}$ is an orthonormal basis of $\g{z}$, and $J_i:=J_{Z_i}$, then
	$\langle [J_i \xi, J_{i+1} \xi], Z_i \rangle = \langle J_i^2 \xi, J_{i+1} \xi \rangle =  0$, for each $i\in\{1,2\}$. 
	However, if $N$ is isomorphic to $N(6,1)$ then $\g{J}\xi$ is not abelian, since $\dim \g{v}=8$, $\dim \g{J}\xi=6$, and hence $J(\g{J}\xi)$ and $\g{J}\xi$ must intersect non-trivially, for any non-zero $J\in\g{J}$. Moreover, in this case, $[\g{J}\xi,(\g{J}\xi)^\perp]\neq 0$, since $\dim\g{J}(\g{J}\xi)^\perp=6=\dim \g{J}\xi$ and therefore $\g{J}(\g{J}\xi)^\perp$ and $\g{J}\xi$ have non-trivial intersection.
\end{remark}

\begin{proposition}\label{prop:bracket_0}
If $[\g{J}\xi,(\g{J}\xi)^\perp]=0$ and $m\geq 2$, then $\g{J}\xi$ is not abelian.
\end{proposition}
\begin{proof}
	Let $V\in(\g{J}\xi)^\perp$, $W\in\g{J}\xi$ and $Z\in\g{z}$. Then, by assumption, $0=\langle [W,V],Z\rangle=\langle J_ZW,V\rangle$, which implies that $\g{J}(\g{J}\xi)\subset \R \xi\oplus\g{J}\xi$. Fix any non-zero $J\in\g{J}$. Hence, by dimension reasons, we have $J(\g{J}\xi)=(\g{J}\xi\ominus \R J\xi)\oplus\R \xi$. Therefore, $\g{J}\xi\oplus\R \xi$ is a $J$-complex subspace of $\g{v}$, where we endow $\g{v}$ with the complex structure $J$. Hence, if $m\geq 2$, $\g{J}\xi$ admits a $J$-complex subspace of the form $\R W\oplus\R JW$, with $W\neq 0$, and then $[W,JW]\neq 0$ by~\eqref{eq:inner_gHg}, which concludes the proof. 
\end{proof}

At this point, we can finish the proof of the main result of this section.

\begin{proof}[Proof of Theorem~\ref{th:Heisenberg}]
	First observe that, by~\cite{Lauret:mathann}, a simply connected nilpotent Lie group with a left-invariant metric is an algebraic Ricci soliton if and only if it is a Ricci soliton. Since $S$ is simply connected and nilpotent (as a codimension one subgroup of $N$), in order to prove the theorem we just have to analyze the algebraic Ricci soliton condition. 
	
	By Proposition~\ref{prop:Heisenberg_charact}, the four items in the statement of Theorem~\ref{th:Heisenberg} give rise to algebraic Ricci solitons. Indeed, in case (i) $\g{J}\xi$ is one-dimensional and, hence, abelian, and also $[\g{J}\xi,(\g{J}\xi)^\perp]=0$ since $\g{J}(\g{J}\xi)=\R\xi$ is orthogonal to $(\g{J}\xi)^\perp$. In case (ii), $(\g{J}\xi)^\perp$ is one-dimensional and $\g{J}\xi$ is abelian as well (see Remark~\ref{rem:dimJxi01}). In case (iii) we have $(\g{J}\xi)^\perp=0$ (again as noticed in Remark~\ref{rem:dimJxi01}), and then $[\g{J}\xi,(\g{J}\xi)^\perp]=0$. Finally, in case (iv), both  $\g{J}\xi$ and $(\g{J}\xi)^\perp$ are abelian by Proposition~\ref{prop:perp_abelian}. 
	
	In order to prove the converse, assume that $S$ is an algebraic Ricci soliton, and let us apply again Proposition~\ref{prop:Heisenberg_charact}. We distinguish two main cases depending on whether $(\g{J}\xi)^\perp$ is abelian or not.
	
	Assume first that $(\g{J}\xi)^\perp$ is abelian. Then, by Proposition~\ref{prop:perp_abelian}, either $\dim (\g{J}\xi)^\perp\in\{0,1\}$ or $N$ corresponds to item (iv) in the statement of Theorem~\ref{th:Heisenberg}. If $(\g{J}\xi)^\perp=0$ (and hence also $[\g{J}\xi,(\g{J}\xi)^\perp]=0$), then, as noted in Remark~\ref{rem:dimJxi01}, $N$ is isomorphic to $N(1,1)$, $N(3,1,0)$ or $N(7,1,0)$, which correspond to items (i) and (iii) in the statement. Again, as observed in Remark~\ref{rem:dimJxi01}, if $\dim(\g{J}\xi)^\perp=1$, then $N$ is isomorphic to $N(2,1)$ or $N(6,1)$. In the first case $\g{J}\xi$ is abelian (and this leads to item (ii) in the statement), whereas in the second case $\g{J}\xi$ is not abelian and $[\g{J}\xi,(\g{J}\xi)^\perp]\neq 0$. 
	
	Finally, assume that $(\g{J}\xi)^\perp$ is not abelian. Therefore, by Proposition~\ref{prop:Heisenberg_charact} we must have that $\g{J}\xi$ is abelian and $[\g{J}\xi,(\g{J}\xi)^\perp]= 0$. Then, Proposition~\ref{prop:bracket_0} implies that $N$ is isomorphic to $N(1,k)$, which leads to item (i) in the statement of Theorem~\ref{th:Heisenberg}.
\end{proof}

\section{Solvsoliton hypersurfaces of Damek-Ricci spaces}\label{sec:DR}
This section is devoted to derive the classification of codimension one Ricci soliton subgroups of Damek-Ricci spaces. In \S\ref{subsec:DR_prelim} we recall the definition and fundamental facts about Damek-Ricci spaces, and in \S\ref{subsec:DR_proof} we prove the classification result contained in Theorem~\ref{th:DR}.

\subsection{Damek-Ricci spaces}\label{subsec:DR_prelim}\hfill

Let $\g{a}$ be a one-dimensional real vector space, $B\in\g{a}$ a non-zero vector, and $\g{n}$ a generalized Heisenberg algebra. We will use the notations and facts introduced in \S\ref{subsec:Heisenberg_prelim} for generalized Heisenberg algebras and groups; in particular, recall that $\g{n}=\g{v}\oplus\g{z}$. We define the vector space direct sum $\g{a}\oplus\g{n}$, which we endow with the Lie algebra structure determined by the Lie algebra structure of the generalized Heisenberg algebra $\g{n}$ and the relations
\[
[B, U]=\frac{1}{2}U, \qquad [B,Z]=Z, \qquad\text{for every } U\in\g{v}, Z\in\g{z}.
\]
This converts $\g{a}\oplus\g{n}$ into a solvable Lie algebra. We equip $\g{a}\oplus\g{n}$ with the positive definite inner product $\langle\cdot,\cdot\rangle$ that extends the inner product of $\g{n}$, makes $B$ into a unit vector and $\g{a}\oplus\g{n}$ into an orthogonal direct sum.
The simply connected solvable Lie group $AN$ with Lie algebra $\g{a}\oplus\g{n}$, endowed with the left-invariant Riemannian metric determined by the inner product $\langle\cdot,\cdot\rangle$ on~$\g{a}\oplus\g{n}$, is called a \emph{Damek-Ricci space}. Analogously as in \S\ref{subsec:Heisenberg_prelim}, we~will use the notation $AN(m,k)$ or $AN(m,k_+,k_-)$ to refer to a Damek-Ricci space whose underlying generalized Heisenberg group with Lie algebra~$\g{n}$ is $N(m,k)$ or $N(m,k_+,k_-)$, respectively.

Damek-Ricci spaces constitute an important family of Einstein solvmanifolds. This~family contains the rank one symmetric spaces of non-compact type and non-constant curvature, that is, the hyperbolic spaces over the complex numbers, the quaternions and the octonions. Indeed, these are the only symmetric Damek-Ricci spaces: $AN(1,k)$, $k\geq 1$,~which is homothetic to a complex hyperbolic space $\C \mathsf{H}^{k+1}$, $AN(3,k,0)\cong AN(3,0,k)$, $k\geq 1$, which is homothetic to a quaternionic hyperbolic space $\H \mathsf{H}^{k+1}$, and $AN(7,1,0)\cong AN(7,0,1)$, which is homothetic to the Cayley hyperbolic plane $\mathbb{O} \mathsf{H}^2$. The real hyperbolic spaces $\R\mathsf{H}^{k+1}$, $k\geq 1$, could be considered as members of this family if one had allowed $\g{z}=0$ (in this~case $\g{n}=\g{v}$ would be abelian). However, in this section we assume that $\dim\g{z}>0$,~as~our problem in real hyperbolic spaces is straightforward (see the proof of Theorem~\ref{th:symmetric} at the~end~of~\S\ref{subsec:symmetric_proof}).

The Levi-Civita connection of a Damek-Ricci space $AN$ is given by the formula
\begin{equation}\label{eq:LC_DR}
	\nabla_{sB+V+Y}\bigl(rB+U+X\bigr)=\left(\frac{1}{2}\langle U,V\rangle +\langle X,Y\rangle\right)\! B-\frac{1}{2}J_XV-\frac{1}{2}J_YU-\frac{1}{2}rV-\frac{1}{2}[U,V]-rY,
\end{equation}
for any $s,r\in\R$, $U,V\in\g{v}$ and $X,Y\in\g{z}$. Using this, one can derive formulas for the curvature tensor and, thus, one can deduce that the Ricci operator is given by
\begin{equation}\label{eq:Ric_DR}
\Ric^{AN}=-\left(m+\frac{n}{4}\right)\id,
\end{equation}
where we recall that $m=\dim\g{z}$ and $n=\dim\g{v}$. For more information on Damek-Ricci spaces, we refer to~\cite{BTV}.

\subsection{Proof of the classification result}\label{subsec:DR_proof}\hfill

Let $S$ be a connected Lie subgroup of codimension one of a Damek-Ricci space $AN$. Let $\g{s}=(\g{a}\oplus\g{n})\ominus\R\xi$ be the corresponding Lie subalgebra of $\g{a}\oplus\g{n}$, where $\xi=aB+U+Z$ is a unit vector, with $a\in\R$, $U\in\g{v}$ and $Z\in\g{z}$. Then $J_Z U$, $\lvert Z\rvert^2 U-\lvert U\rvert^2Z\in\g{s}\cap\g{n}$ and, by~\eqref{eq:xi_H_DR} (with $W=U$) and $[\g{v},\g{v}]=\g{z}$, we deduce that $Z=0$. Therefore
\[
\xi=aB+U, \qquad \text{with}\quad  a\in\R,\quad U\in\g{v},\quad a^2+|U|^2=1.
\]
Note that, in this case, we have the orthogonal decompositions
\[
\g{s}=\R(|U|^2B-aU)\oplus(\g{v}\ominus\R U)\oplus\g{z}=\R(|U|^2B-aU)\oplus\g{J}U\oplus(\g{J}U)^\perp\oplus\g{z},
\]
where, as in~\S\ref{subsec:Heisenberg_proof}, $\g{J}=\{J_Z:Z\in\g{z}\}$, $\g{J}U=\{J_Z U:Z\in\g{z}\}$ and $(\g{J}U)^\perp=\g{v}\ominus(\R U\oplus\g{J}U)$.

Let us calculate the shape operator of $S$ by making use of~\eqref{eq:LC_DR}: 
\begin{align}\label{eq:shape_DR}
	\nonumber
\Ss_\xi (|U|^2B-aU)&=-\nabla_{|U|^2B-aU}(aB+U)=\frac{a}{2}(|U|^2B-aU), \\ \nonumber
\Ss_\xi V&=-\nabla_V(aB+U)=\frac{a}{2}V, & & \text{for any } V\in(\g{J}U)^\perp, \\
\Ss_\xi J_Z U &= -\nabla_{J_Z U}(aB+U)=\frac{1}{2}(aJ_ZU+|U|^2Z), & & \text{for any }Z\in\g{z},\, J_ZU\in \g{J}U, \\ \nonumber
\Ss_\xi Z&=-\nabla_Z (aB+U)=\frac{1}{2}J_ZU+aZ, & & \text{for any } Z\in\g{z}.
\end{align}
In particular
\begin{equation}\label{eq:trS_DR}
\tr\Ss_\xi=\frac{a}{2}+\frac{a}{2}(n-m-1)+\frac{a}{2}m+am=a\biggl(m+\frac{n}{2}\biggr).
\end{equation}

In order to calculate the normal Jacobi operator $R_\xi=R^{AN}(\cdot,\xi)\xi$ of $S$, we can combine the definition of the curvature tensor of the ambient space, $R^{AN}(X,Y)=[\nabla_X,\nabla_Y]-\nabla_{[X,Y]}$, with the formula~\eqref{eq:LC_DR} for the Levi-Civita connection, or alternatively use the formula for the Jacobi operator of a Damek-Ricci space~\cite[\S4.1.8]{BTV}. In any case, we get
\begin{equation}\label{eq:jacobi_DR}
\begin{aligned}
R_\xi (|U|^2B-aU)&=-\frac{1}{4}(|U|^2B-aU), 
\\
R_\xi V&=-\frac{1}{4}V, &\quad &\text{for any } V\in(\g{J}U)^\perp,
\\
R_\xi J_Z U &= -\frac{1}{4}\bigl(3|U|^2+1\bigr)J_ZU-\frac{3}{4}a|U|^2Z, &\quad &\text{for any }Z\in\g{z},\, J_ZU\in \g{J}U, 
\\
R_\xi Z&=-\frac{3}{4}aJ_ZU+\biggl(\frac{3}{4}|U|^2-1\biggr)Z, &\quad &\text{for any } Z\in\g{z}.
\end{aligned}
\end{equation}
Therefore, by inserting \eqref{eq:Ric_DR}, \eqref{eq:shape_DR}, \eqref{eq:trS_DR} and~\eqref{eq:jacobi_DR} into~\eqref{eq:Ric_Gauss}, recalling that $a^2+|U|^2=1$, and setting for convenience
\begin{equation}\label{eq:mn_tilde}
\widetilde{m}:=-m-\frac{n}{4}\qquad \text{and}\qquad \widetilde{n}:=m+\frac{n}{2},
\end{equation}
we obtain that the Ricci operator of the hypersurface $S$ is given by:
\begin{align}\label{eq:Ric_DR_hyp}
\begin{split} 
	\Ric(V)&=\frac{1}{4}\left(\lvert U\rvert^2+4\wt{m}+2a^2\wt{n}\right) V, 
	\\ 
	\Ric (J_ZU)&=\frac{1}{4}\left(3\lvert U\rvert^2+4\wt{m}+2a^2\wt{n}\right)J_ZU+\frac{\wt{n}}{2}a\lvert U\rvert^2 Z,
	\\
	\Ric(Z)&=\frac{\wt{n}}{2} a J_ZU +(\wt{m}+a^2\wt{n})Z,
\end{split} 
\end{align}
for any $V\in\R(|U|^2B-aU)\oplus(\g{J}U)^\perp$ and $Z\in\g{z}$. 

In what follows we will study when $S$ is an algebraic Ricci soliton, which amounts to determining under which circumstances and for which values of $c\in\R$ the endomorphism $D:=\Ric-c\id$ of $\g{s}$ is a derivation of $\g{s}$.

\begin{proposition}\label{prop:DR_2cases}
	If $S$ is an algebraic Ricci soliton, then $\xi\in\g{a}$ or $\xi\in\g{v}$. Moreover, if $\xi\in\g{v}$ then $c=(1+4\wt{m})/4$.
\end{proposition}
\begin{proof}
	If $S$ is an algebraic Ricci soliton, then $D$ must, in particular, satisfy
	\[
	D[|U|^2B-aU,Z]=[D(|U|^2B-aU),Z]+[|U|^2B-aU,DZ],
	\]
	for any $Z\in\g{z}$. By the definition of $D$ and~\eqref{eq:Ric_DR_hyp}, the term on the left-hand side turns out to be
	\[
	D[|U|^2B-aU,Z]=|U|^2DZ=|U|^2\frac{\wt{n}}{2} a J_ZU + 
 |U|^2 \left( \wt{m}+a^2\wt{n}-c \right) Z.
	\]
	Similarly, for the addends of the right-hand side we have:
	\begin{align*}
		\left[D(|U|^2B-aU),Z\right] &= \frac{1}{4}\left(\lvert U\rvert^2+4\wt{m}+2a^2\wt{n}-4c\right)|U|^2 Z,
		\\
		\left[|U|^2B-aU,DZ\right]&=
		\frac{\wt{n}}{4} a |U|^2 J_ZU +|U|^2 \left(\wt{m}-c+\frac{\wt{n}}{2} a^2\right)Z.
	\end{align*}
	Altogether, considering the component in the direction of $J_Z U$, we get from the first equation of the proof that either $U=0$ or $a=0$, which proves the first claim in the statement. Finally, if $\xi\in\g{v}$, then $a=0$ and $|U|=1$, and taking the $Z$-component in the equation considered above we get the expression for $c$ in the statement.
\end{proof}

Note that, if $\xi=U\in\g{v}$, then \eqref{eq:Ric_DR_hyp} translates into
\begin{equation}\label{eq:Ric_DR_2}
\Ric\rvert_{\g{a}\oplus(\g{J}\xi)^\perp}=\biggl(\wt{m}+\frac{1}{4}\biggr)\id, \qquad \Ric\rvert_{\g{J}\xi}=\biggl(\wt{m}+\frac{3}{4}\biggr)\id,\qquad
\Ric\rvert_{\g{z}}=\wt{m}\id.
\end{equation}
We can now conclude the proof of the main result of this section.

\begin{proof}[Proof of Theorem~\ref{th:DR}]
	First note that, since $\g{a}\oplus\g{n}$ is completely solvable (i.e.\ $\ad(X)$ has only real eigenvalues, for each $X\in\g{a}\oplus\g{n}$), the same is true for the Lie subalgebra $\g{s}=(\g{a}\oplus\g{n})\ominus\R\xi$. 
Hence, by~\cite{Lauret:crelle}, the connected, simply connected Lie subgroup $S$ of a Damek-Ricci space $AN$ with Lie algebra $\g{s}$ is a Ricci soliton if and only if it is an algebraic Ricci soliton.
	
	Now, observe that the examples mentioned in the statement are indeed algebraic Ricci solitons. On the one hand, any generalized Heisenberg group $N$ is an algebraic Ricci soliton, as remarked in~\S\ref{subsec:Heisenberg_prelim}. On the other hand, if $S$ is a Lohnherr hypersurface in $\C \mathsf{H}^2$, then it follows from~\eqref{eq:Ric_DR_hyp} that $D=\Ric-c\id$, with $c=(1+4\wt{m})/4$, has eigenvalues $0$, $1/2$ and $-1/4$ with respective $1$-dimensional eigenspaces $\g{a}$, $\g{J}\xi$ and $\g{z}$. One can easily check that $D$ is a derivation of $\g{s}=\g{a}\oplus(\g{v}\ominus\R\xi)$, which shows that $S$ is an algebraic Ricci soliton.
	
	Finally, we prove the converse implication. Thus, assume that $S$ is an algebraic Ricci soliton. By Proposition~\ref{prop:DR_2cases} we have that $\xi\in\g{a}$ or $\xi\in\g{v}$. 
	
	If $\xi\in\g{a}$, then $\g{s}=\g{n}$ 
	is a generalized Heisenberg algebra. Hence, the corresponding connected Lie group $S$ coincides with the generalized Heisenberg group $N$.

Let us assume that $\xi\in\g{v}$. By Proposition~\ref{prop:DR_2cases} we have $c=(1+4\wt{m})/4$. 
Then by~\eqref{eq:Ric_DR_2} we have that $D = \Ric - c \id$ satisfies 
\[ 
D \rvert_{\g{a}\oplus(\g{J}\xi)^\perp} = 0 , \qquad 
D \rvert_{\g{J}\xi} = \frac{1}{2} \id , \qquad 
D \rvert_{\g{z}} = - \frac{1}{4} \id . 
\] 
Since $D$ is a derivation, 
one can see that $(\g{J}\xi)^\perp$ and $\g{J}\xi$ are abelian subspaces of $\g{v}$, and $[(\g{J}\xi)^\perp,\g{J}\xi]=0$.
But then, by Proposition~\ref{prop:bracket_0} we must have $m=\dim\g{z}=1$, which implies that $AN$ is isometric to a complex hyperbolic space. Moreover, by Proposition~\ref{prop:perp_abelian} we also have $\dim(\g{J}\xi)^\perp\in\{0,1\}$, which forces $AN$ to be in fact isometric to a complex hyperbolic plane $\C \mathsf{H}^2$. Since $\xi\in\g{v}$, $S$ is a Lohnherr hypersurface in $\C \mathsf{H}^2$. This concludes the proof.
\end{proof}

\section{Solvsoliton hypersurfaces of symmetric spaces of non-compact type}\label{sec:symmetric}
	
This section is devoted to studying which codimension one subgroups of the solvable part of the Iwasawa decomposition of the isometry group of a symmetric space of non-compact type are Ricci solitons. In \S\ref{subsec:symmetric_prelim} we gather some preliminaries and useful results regarding symmetric spaces of non-compact type, whereas \S\ref{subsec:symmetric_proof} contains the proof of Theorem~\ref{th:symmetric}.

\subsection{Symmetric spaces of non-compact type}\label{subsec:symmetric_prelim}\hfill

Let $M$ be a symmetric space of non-compact type. Then $M$ can be identified with a quotient space $G/K$, where $G$ is the connected component of the identity of the isometry group of $M$ (up to a finite quotient), and $K$ is the isotropy subgroup of $G$ corresponding to some point $o\in M$. Let $\theta$ be the corresponding Cartan involution of the Lie algebra $\g{g}$ of $G$, and $\g{g}=\g{k}\oplus\g{p}$ the associated Cartan decomposition, where $\g{k}$ and $\g{p}$ are the $(+1)$ and $(-1)$-eigenspaces of $\theta$, respectively. The Lie algebra $\g{g}$ is real semisimple, which implies that its Killing form $B$ is non-degenerate. Indeed, the Cartan decomposition $\g{g}=\g{k}\oplus\g{p}$ is $B$-orthogonal, $B\rvert_{\g{k}\times\g{k}}$ is negative definite, and $B\rvert_{\g{p}\times\g{p}}$ is positive definite (since $M$ is of non-compact type). Hence, defining 
\[
\langle X, Y \rangle_{B_{\theta}} :=-B(\theta X, Y), \qquad \text{for all }X,Y\in\g{g},
\]
we have that $\langle \cdot, \cdot \rangle_{B_{\theta}}$ is a positive definite inner product on $\g{g}$. It is easy to check that this inner product satisfies 
\begin{equation}\label{eq:cartan:inner}
\langle \ad (X)Y, Z \rangle_{B_{\theta}}=- \langle Y, \ad(\theta X)Z \rangle_{B_{\theta}},\qquad \text{for all }X,Y,Z\in\g{g}.
\end{equation}

Let $\g{a}$ be a maximal abelian subspace of $\g{p}$. The rank of $M$ is precisely the dimension~of~$\g{a}$. Then $\{\ad(H):H\in\g{a}\}$ constitutes a commuting family of self-adjoint endomorphisms of $\g{g}$, and hence they diagonalize simultaneously. Their common eigenspaces are called the restricted root spaces of $\g{g}$, whereas their non-zero eigenvalues (which depend linearly on $H\in\g{a}$) are the restricted roots of $\g{g}$. In other words, if for each covector $\lambda\in\g{a}^*$ we define 
\[
\g{g}_\lambda:=\{X\in\g{g}: [H,X]=\lambda(H)X \text{ for all } H\in\g{a}\},
\]
then any $\g{g}_\lambda\neq 0$ is a restricted root space, and any $\lambda\neq 0$ such that $\g{g}_\lambda\neq 0$ is a restricted root. Notice that $\g{g}_0$ is always non-zero, since $\g{a}\subset \g{g}_0$. If  $\Sigma=\{\lambda\in\g{a}^*:\lambda\neq 0,\,\g{g}_\lambda\neq 0\}$ denotes the set of restricted roots, then the  $\langle \cdot, \cdot \rangle_{B_{\theta}}$-orthogonal decomposition 
\begin{equation*}\label{eq:root_space_decomposition}
\g{g}=\g{g}_0\oplus\biggl(\bigoplus_{\lambda\in\Sigma}\g{g}_\lambda\biggr)
\end{equation*}
is called the restricted root space decomposition of $\g{g}$. Moreover, we have the bracket relations $[\g{g}_{\lambda},\g{g}_{\mu}] \subset \g{g}_{\lambda+\mu}$
and $\theta \g{g}_\lambda=\g{g}_{-\lambda}$, for any $\lambda$, $\mu\in\g{a}^*$. 
We also have the $\langle \cdot, \cdot \rangle_{B_{\theta}}$-orthogonal decomposition $\g{g}_0=\g{k}_0\oplus\g{a}$, where $\g{k}_0=\g{g}_0\cap\g{k}$ is the normalizer of $\g{a}$ in $\g{k}$.

For each $\lambda\in \Sigma$, we define $H_\lambda\in\g{a}$ by the relation $B(H_\lambda,H)=\lambda(H)$, for all $H\in\g{a}$. We can introduce an inner product on $\g{a}^*$ given by $\langle \lambda,\mu\rangle:=B(H_\lambda,H_\mu)$. We will write $| \lambda | = \langle \lambda,\lambda\rangle^{1/2}$ for the induced norm on $\g{a}^*$. Then, we have that $\Sigma$ is an abstract root system in $\g{a}^*$. Thus, we can define a positivity criterion on $\g{a}^*$, which allows us to decompose $\Sigma=\Sigma^+\cup(-\Sigma^+)$, where $\Sigma^+$ is the set of positive roots. We will denote by $\Pi\subset \Sigma^+$ the corresponding set of simple roots, that is, those positive roots $\lambda\in \Sigma^+$ which are not sum of two positive roots. Then, any root $\lambda\in\Sigma$ is a linear combination of elements of $\Pi$ with integer coefficients, where all of them are either non-negative (if $\lambda\in\Sigma^+$) or non-positive (if $\lambda\in-\Sigma^+$). We will also make use of the so-called Cartan integers, that is, the integers
\[
A_{\alpha, \lambda} := \frac{2 \langle \alpha, \lambda \rangle}{|\alpha|^2},
\]
for each $\alpha$, $\lambda \in \Sigma$. We refer to ~\cite{K} for more details on root systems.


A fundamental result for our work is the Iwasawa decomposition theorem. At the Lie algebra level, it states that $\g{g}=\g{k}\oplus\g{a}\oplus\g{n}$ is a vector space direct sum, where $\g{n}=\bigoplus_{\lambda\in\Sigma^+}\g{g}_\lambda$. Note that $\g{n}$ and $\g{a}\oplus\g{n}$ are nilpotent and solvable Lie subalgebras of $\g{g}$, respectively; indeed $[\g{a} \oplus \g{n}, \g{a} \oplus \g{n}] \subset \g{n}$. 
Let $A$, $N$ and $AN$ denote the (closed) connected Lie subgroups of $G$ with Lie algebras $\g{a}$, $\g{n}$ and $\g{a}\oplus\g{n}$, respectively. The Iwasawa decomposition at the Lie group level states that $G$ is diffeomorphic to the Cartesian product $K\times A\times N$. It follows that $AN$ acts simply transitively on $M\cong G/K$, and hence $M$ is diffeomorphic to the solvable Lie group $AN$. If we pull back the symmetric Riemannian metric on $M$ to $AN$ via this diffeomorphism, we obtain a left-invariant metric on $AN$. Thus, any symmetric space of non-compact type $M\cong G/K$ is isometric to the solvable part $AN$ of an Iwasawa decomposition for $G$, equipped with a certain left-invariant metric. We will denote by $\langle\cdot,\cdot\rangle$ this left-invariant metric on $AN$, and also the corresponding inner product~on~$\g{a}\oplus\g{n}$. 

In what follows, the symmetric space $M$ will be assumed to be irreducible.
If $X,Y\in\g{a}\oplus\g{n}$, and denoting orthogonal projections (with respect to $\langle\cdot,\cdot\rangle_{B_\theta}$) with subscripts, the following relation between  $\langle\cdot,\cdot\rangle$ and  $\langle\cdot,\cdot\rangle_{B_\theta}$ holds (up to homothety of the metric on $M$):
\begin{equation}\label{eq:relation:inner}
\langle X,Y\rangle =\langle X_\mathfrak{a},Y_\mathfrak{a}\rangle_{B_\theta}+\frac{1}{2}\langle X_\mathfrak{n},Y_\mathfrak{n}\rangle_{B_\theta}.
\end{equation}
Using Koszul's formula and relations~\eqref{eq:relation:inner}-\eqref{eq:cartan:inner}, one can obtain a formula for the Levi-Civita connection $\nabla$ of the Lie group $AN$ with its left-invariant metric. Indeed, if $X,Y,Z\in\g{a}\oplus\g{n}$:
\begin{equation}\label{eq:inner:an:bphi}
\langle \nabla_X Y,Z\rangle =\frac{1}{4}\langle [X,Y]+[\theta X,Y]-[X,\theta Y],Z\rangle_{B_\theta}.
\end{equation}
We warn of the use of two different inner products on the previous formula.

We end this subsection with two lemmas that will be useful in what follows. We also note that, from now on, whenever we consider a unit vector $X\in\g{a}\oplus\g{n}$ we will understand that it has length one with respect to the inner product $\langle\cdot,\cdot\rangle$ defined on $\g{a}\oplus\g{n}$.
\begin{lemma}\label{lemma:berndt:sanmartin} \emph{\cite[Lemma 2.3]{BS18}}
	Let $\lambda \in \Sigma^{+}$ and $X,Y \in \g{g}_{\lambda}$ be orthogonal. Then:
	\begin{enumerate}[{\rm (i)}]
		\item $[\theta X, X] = 2\langle X, X \rangle H_{\lambda}=\langle X, X \rangle_{B_{\theta}} H_{\lambda}$. \label{lemma:berndt:sanmartin:i}
		\item $[\theta X, Y] \in \g{k}_0 = \g{g}_0 \ominus \g{a}$. \label{lemma:berndt:sanmartin:ii}
	\end{enumerate}
\end{lemma}

\begin{lemma} \label{lemma:xi:calculations}
	Let $\xi = a H_{\alpha} + b X_{\alpha}$ be a unit vector, where $\alpha \in \Pi$ is a simple root, $X_{\alpha} \in \g{g}_{\alpha}$ is a unit vector, and $a$, $b \in \R$. Then:
	\begin{enumerate}[\rm(i)]
		\item $[\theta \xi, \xi] = -ab |\alpha|^2 X_{\alpha} + ab |\alpha|^2 \theta X_{\alpha} +2 b^2 H_{\alpha}$. \label{lemma:xi:calculations:i}
		\item $\nabla_{\xi} \xi = b^2 H_{\alpha} - ab |\alpha|^2 X_{\alpha} $. \label{lemma:xi:calculations:ii}
		\item $\nabla_H X = 0$ for all $H \in \g{a}$ and $X \in \g{a} \oplus \g{n}$.\label{lemma:xi:calculations:iii}
	\end{enumerate}
\end{lemma}
\begin{proof}
	First, using the facts that $\g{a}$ is abelian, $\theta \rvert_{\g{a}} = - \id$, $\theta \g{g}_{\lambda} = \g{g}_{-\lambda}$ for all $\lambda \in \Sigma$, the definition of restricted root space, and Lemma~\ref{lemma:berndt:sanmartin}~(\ref{lemma:berndt:sanmartin:i}), we deduce
	\begin{align*}
	[\theta \xi, \xi] & = [\theta (a H_{\alpha} + b X_{\alpha}), a H_{\alpha} + b X_{\alpha}] = -ab[H_{\alpha}, X_{\alpha}] + ab [\theta X_{\alpha}, H_{\alpha}] + b^2 [\theta X_{\alpha}, X_{\alpha}]\\
	& = -ab |\alpha|^2 X_{\alpha} + ab |\alpha|^2 \theta X_{\alpha} +2 b^2 H_{\alpha},
	\end{align*}
	which proves assertion~(\ref{lemma:xi:calculations:i}). Now, if $Z \in \g{a} \oplus \g{n}$, using~\eqref{eq:inner:an:bphi}, assertion~(\ref{lemma:xi:calculations:i}), and~\eqref{eq:relation:inner}, we get 
	\begin{align*}
	\langle \nabla_{\xi} \xi, Z \rangle & = \frac{1}{4} \langle [\xi, \xi] + [\theta \xi, \xi] - [\xi,\theta \xi]  , Z \rangle_{B_{\theta}} 
	\\
	& = \frac{1}{2} \langle [\theta \xi, \xi], Z \rangle_{B_\theta}
	= \frac{1}{2} \langle -ab |\alpha|^2 X_{\alpha} + ab |\alpha|^2 \theta X_{\alpha} +2 b^2 H_{\alpha} , Z \rangle_{B_{\theta}} \\
	& = -\frac{1}{2}ab |\alpha|^2 \langle  X_{\alpha}, Z \rangle_{B_\theta}+  b^2\langle   H_{\alpha} , Z \rangle_{B_{\theta}} = -ab |\alpha|^2  \langle  X_{\alpha}, Z \rangle + b^2  \langle   H_{\alpha} , Z \rangle.
	\end{align*}
	This proves assertion~(\ref{lemma:xi:calculations:ii}). Finally, using again~\eqref{eq:inner:an:bphi}, we have
	\begin{align*}
	\langle \nabla_{H} X, Z \rangle & = \frac{1}{4} \langle [H, X] + [\theta H, X] - [H, \theta X]  , Z \rangle_{B_{\theta}} = 0, 
	\end{align*}
	for all $X$, $Z \in \g{a} \oplus \g{n}$ and all $H \in \g{a}$. This finishes the proof.
\end{proof}

\subsection{Proof of the classification result}\label{subsec:symmetric_proof}\hfill

Let $M\cong G/K$ be an irreducible symmetric space of non-compact type, and let $AN$ be the solvable part of an Iwasawa decomposition of $G$. 
Let $S$ be a codimension one subgroup of $AN$ with Lie algebra $\g{s}$. Then there exists a unit vector $\xi \in \g{a} \oplus \g{n}$ such that $\g{s} =  (\g{a} \oplus \g{n}) \ominus \R \xi$. We have (see~\cite[Lemma~5.3]{BT:jdg}):
\begin{lemma}\label{lemma:JDG}
If  $\g{s} =  (\g{a} \oplus \g{n}) \ominus \R \xi$ is a Lie subalgebra of $\g{a}\oplus\g{n}$, then  $\xi\in\g{a}$ or $\xi\in\R H_\alpha\oplus\g{g}_\alpha$ for some simple root $\alpha\in \Pi$.
\end{lemma}

The subgroup $S$ of $AN$, with the induced metric, is by definition an algebraic Ricci soliton if the Ricci tensor $\Ric$ of $S$ can be written as $\Ric = c\id + D$, where $c \in \R$ and $D$ is a derivation of $\g{s}$. 
Since an irreducible symmetric space of non-compact type is Einstein with negative scalar curvature, the Ricci operator of $M$ satisfies $\Ric^M = k \id$ for some $k<0$. Therefore, by~\eqref{eq:Ric_Gauss}, studying when $S$ is an algebraic Ricci soliton is equivalent to determining when the endomorphism 
\begin{equation}\label{eq:operator:d}
D := \tr (\Ss_{\xi}) \Ss_{\xi} - (R_{\xi} + \Ss_{\xi}^2) + c  \id
\end{equation}
of $\g{s}$ is a derivation, for some $c\in\R$. 

Lemma~\ref{lemma:JDG} indicates that we have to consider two cases. The first one (corresponding to $\xi\in\g{a}$) has been addressed in~\cite[Proposition~5.2]{CHKTT18}, but we include a direct proof here for the sake of completeness.

\begin{remark}\label{rem:horospheres}
	The orbits of the connected Lie subgroup $S$ of $AN$ with Lie algebra $\g{s}=(\g{a}\ominus \R H)\oplus\g{n}$, for some unit vector $H$ in the  Weyl chamber $\{H\in\g{a}:\lambda(H)\geq 0 \text{ for all }\lambda\in\Sigma^+\}$, are horospheres of the symmetric space $M$, meaning that they are level sets of a~Busemann function~\cite[\S1.10]{eberlein}. Indeed, from Lemma~\ref{lemma:xi:calculations}~(\ref{lemma:xi:calculations:ii}) we get that $H$ is a geodesic vector field. Let $\gamma$ be one of the geodesic integral curves of $H$. As $H$ is left-invariant, any~other integral curve of $H$ is of the form $g\circ\gamma$, for some $g\in AN$. 
	The group $AN$ is contained in the parabolic subgroup of $G$ given by the stabilizer of the point at infinity $\lim_{t\to\infty}\gamma(t)$ in the ideal boundary of $M$ (see \cite[\S2.17]{eberlein}).
	Hence, all integral curves~of~$H$ are geodesics asymptotic to such common point at infinity. Consider the Busemann function $f_H(p)=\lim_{t\to+\infty}(d(p,\gamma(t))-t)$ on $M$ associated with~$\gamma$. By \cite[\S1.10.2(2)]{eberlein}, we have $\mathrm{grad}\, f_H=-H$ everywhere. Thus, the level sets of $f_H$ are everywhere orthogonal to $H$, which implies that they are precisely the $S$-orbits. Finally, note that the $S$-orbits are mutually congruent~(cf.~\cite{BT:jdg}).
\end{remark}

\begin{proposition}\label{prop:horosphere}
	Let $M$ be an irreducible symmetric space of non-compact type with solvable Iwasawa group $AN$. Let $S$ be the connected Lie subgroup of $AN$ with Lie algebra $\g{s} = (\g{a} \ominus \R H) \oplus \g{n}$, for some unit vector $H \in \g{a}$. Then $S$ is a solvsoliton.
\end{proposition}
\begin{proof}
Let $X \in \g{n}$ and $Y\in\g{a}\oplus\g{n}$. Then, using~\eqref{eq:inner:an:bphi} and~\eqref{eq:relation:inner} we have
\begin{align*}
\langle \nabla_X H, Y \rangle & = \frac{1}{4} \langle [X, H] + [\theta X, H] - [X, \theta H], Y \rangle_{B_\theta}  = - \frac{1}{2} \langle [H, X] , Y \rangle_{B_\theta} = -\langle [H, X] , Y \rangle.
\end{align*}
This means that $\Ss_{H} X =  \ad(H) X$ for all $X \in \g{n}$. Now, using this and Lemma~\ref{lemma:xi:calculations}~(\ref{lemma:xi:calculations:iii}) twice, we have
\begin{align*}
R_H (X) = \nabla_X \nabla_H H - \nabla_H \nabla_X H - \nabla_{[X, H]} H = - \Ss_H [H, X] = -\ad(H)^2 X,
\end{align*} 
for all $X \in \g{n}$. Moreover, using Lemma~\ref{lemma:xi:calculations}~(\ref{lemma:xi:calculations:iii}) we easily get  $\Ss_{H} \rvert_{\g{a} \ominus \R H } = R_{H} \rvert_{\g{a} \ominus \R H} = 0$. All in all, we deduce  $\Ss_H=  \ad(H)$ and $R_H=-\ad(H)^2$. Therefore, we get 
\[
\Ric = \Ric^M +  \tr (\Ss_{H}) \Ss_{H} - R_{H} - \Ss_{H}^2 = k\id +  \tr (\ad(H)) \ad(H) ,
\]
which is precisely the sum of a derivation of $\g{s}$ and a multiple of the identity.
\end{proof}

\medskip
From now one we will focus on dealing with the second case in Lemma~\ref{lemma:JDG}. More precisely, the rest of this section will be devoted to prove the following:

\begin{theorem}\label{th:xi_diagonal}
	Let $\xi = a H_{\alpha} + b X_{\alpha}$ be a unit vector, where $X_{\alpha} \in \g{g}_{\alpha}$ is a unit vector, $\alpha \in \Pi$ a simple root, and $a$, $b \in \R$, $b\neq 0$. Let $S$ be the connected Lie subgroup of $AN$ with Lie algebra $\g{s} = (\g{a} \oplus \g{n}) \ominus \R \xi$. If the hypersurface $S$, with the induced metric, is an algebraic Ricci soliton, then $\dim \g{a} = 1$, that is, $M$ is a rank one symmetric space.
\end{theorem}

Thus, we will assume from now on that $M$ is an irreducible symmetric space of non-compact type and rank greater than one, $\g{s} =  (\g{a} \oplus \g{n}) \ominus \R \xi$ is a Lie subalgebra  of $\g{a}\oplus\g{n}$ with $\xi = a H_{\alpha} + b X_{\alpha}$ for some simple root $\alpha \in \Pi$, some unit vector $X_{\alpha} \in \g{g}_{\alpha}$, and $a$, $b \in \R$. Note that, since $\xi$ and $X_{\alpha}$ are unit vectors, we also have that $a^2 |\alpha|^2 + b^2 =1$. Our goal will be to prove that, if $S$ is an algebraic Ricci soliton (that is, if the endomorphism $D$ given in~\eqref{eq:operator:d} is a derivation), then $b = 0$. 

Let us consider the vector 
\[
U := b |\alpha|^{-1} H_{\alpha} -a |\alpha| X_{\alpha}.
\] Note that $\langle U, U \rangle = 1$ and $\langle U, \xi \rangle = 0$. Thus, we can consider the $\langle\cdot,\cdot\rangle$-orthogonal decomposition $\g{s} = (\g{a} \ominus \R H_{\alpha}) \oplus (\g{n} \ominus \R X_{\alpha}) \oplus \R U$ of the tangent space to $S$ at the identity.

Our efforts will be focused on calculating the endomorphism $D$ of $\g{s}$. Although one could determine it completely, it will be enough for our purposes to just calculate $D$ when restricted to a particular subspace.
Actually, we will select a root $\beta \in \Sigma^{+} \backslash \{ \alpha \}$ such that $\beta + \alpha\in\Sigma$ and $\beta - \alpha\notin\Sigma$. Such a root exists since $M$ is irreducible and has rank greater than one by assumption: it suffices to choose a simple root $\beta$ connected to $\alpha$ in the Dynkin diagram of $M$. Then, the main part of what follows will be devoted to determining the restriction of the endomorphism $D$ to the subspace
\begin{equation}\label{eq:decomposition:tangent}
(\g{a} \ominus \R H_{\alpha}) \oplus (\g{g}_{\beta} \oplus \g{g}_{\beta + \alpha}) \oplus \R U\subset\g{s}.
\end{equation} 

We start by analyzing the restriction of $D$ to $\g{a} \ominus \R H_{\alpha}$.
	
\begin{proposition}\label{proposition:d:a}
Let $H \in \g{a} \ominus \R H_{\alpha}$. Then $D H = c H$.
\end{proposition}

\begin{proof}
Let $H\in\g{a}\ominus\R H_\alpha$ be arbitrary. From Lemma~\ref{lemma:xi:calculations}~(\ref{lemma:xi:calculations:iii}) we obtain $\Ss_{\xi} H = - \nabla_{H} \xi = 0$, and hence $\Ss_{\xi}\vert_{\g{a}\ominus \R H_{\alpha}}=\Ss_{\xi}^2\vert_{\g{a}\ominus \R H_{\alpha}}=0$. Note now that $[H, \xi] = [H, a H_\alpha + b X_{\alpha}] = 0$, since $H \in \g{a} \ominus \R H_{\alpha}$. Using this and Lemma~\ref{lemma:xi:calculations}~(\ref{lemma:xi:calculations:iii}) twice, we get
\begin{align*}
R_{\xi} (H) = R^{AN} (H, \xi) \xi  = \nabla_{H} \nabla_{\xi} \xi - \nabla_{\xi} \nabla_{H} \xi - \nabla_{[H, \xi]} \xi = 0
\end{align*}
for the Jacobi operator. All in all, from~\eqref{eq:operator:d} we conclude
\[
D H = (\tr (\Ss_{\xi}) \Ss_{\xi}  - R_{\xi} - \Ss_{\xi}^2  + c \id ) H = c H.\qedhere
\]
\end{proof}

We have just determined $D\vert_{\g{a}\ominus \R H_{\alpha}}$. Before continuing with the calculation of $D$, we will introduce two results that will allow us to use the Levi-Civita connection more efficiently.

\begin{lemma}\label{lemma:levi:civita:connection:n}
Let $X = X_\g{a} + \sum_{\gamma \in \Sigma^{+}} X_{\gamma}$ and $Y = Y_\g{a} + \sum_{\gamma \in \Sigma^{+}} Y_{\gamma}$ be vectors in $\g{a} \oplus \g{n}$, with $X_\g{a}$, $Y_\g{a} \in \g{a}$ and $X_{\gamma}$, $Y_{\gamma} \in \g{g}_{\gamma}$ for all $\gamma \in \Sigma^{+}$. If $\langle X_{\gamma}, Y_{\gamma} \rangle = 0$ for all $\gamma \in \Sigma^{+}$, then
\begin{samepage}
\[
\nabla_X Y = \frac{1}{2} ([X, Y] + [\theta X, Y] - [X, \theta Y])_{\g{n}},
\]
where $(\cdot)_{\g{n}}$ denotes the orthogonal projection onto the Lie subalgebra $\g{n}$.
\end{samepage}
\end{lemma}
\begin{proof}
First of all, using~\eqref{eq:cartan:inner} and the assumption that $\langle X_{\gamma}$, $Y_{\gamma} \rangle = 0$ for all $\gamma \in \Sigma^{+}$, we~have
\begin{align*}
\langle [X, \theta Y], H \rangle_{B_\theta} & = \langle X, [Y, H] \rangle_{B_\theta} = - \langle X, [H, Y_\g{a} + \sum_{\gamma \in \Sigma^{+}} Y_{\gamma}] \rangle_{B_\theta} = -\sum_{\gamma \in \Sigma^{+}} \gamma(H) \langle X,   Y_{\gamma}\rangle_{B_\theta}  \\
&= -\sum_{\gamma \in \Sigma^{+}} \gamma(H) \langle X_\g{a}, Y_{\gamma} \rangle_{B_\theta} -\sum_{\gamma \in \Sigma^{+}} \sum_{\lambda \in \Sigma^{+}} \gamma(H) \langle X_\lambda,   Y_{\gamma}\rangle_{B_\theta} = 0,
\end{align*}
for all $H \in \g{a}$. This means that $[X, \theta Y]\in\g{g}$ is $\langle\cdot,\cdot\rangle_{B_\theta}$-orthogonal to $\g{a}$. Analogously, $[\theta X, Y]\in\g{g}$ is also $\langle\cdot,\cdot\rangle_{B_\theta}$-orthogonal to $\g{a}$. Since $[X, Y] \in[\g{a}\oplus\g{n},\g{a}\oplus\g{n}]=\g{n}$, then $[X, Y] + [\theta X, Y]-[X,\theta Y]$ is $\langle\cdot,\cdot\rangle_{B_\theta}$-orthogonal to $\g{a}$. Now, using this, together with~\eqref{eq:inner:an:bphi} and \eqref{eq:relation:inner}, we deduce 
\begin{align*}
\langle \nabla_X Y, Z \rangle &= \frac{1}{4} \langle [X, Y] + [\theta X, Y] - [X, \theta Y], Z   \rangle_{B_\theta} = \frac{1}{4} \langle ([X, Y] + [\theta X, Y] - [X, \theta Y])_\g{n}, Z_\g{n}  \rangle_{B_\theta}\\
&= \frac{1}{2} \langle ([X, Y] + [\theta X, Y] - [X, \theta Y])_\g{n}, Z  \rangle
\end{align*}
for all $Z \in \g{a} \oplus \g{n}$. This concludes the proof.
\end{proof}

\begin{lemma}\label{lemma:projection:xi:y}
Let $\xi = a H_{\alpha} + b X_{\alpha}$, where $\alpha \in \Pi$, $X_{\alpha} \in \g{g}_{\alpha}$, and $a$, $b \in \R$. Let $Y_{\lambda} \in \g{g}_{\lambda}$ be a vector orthogonal to $\xi$, for some $\lambda \in \Sigma^{+}$. Then:
\begin{enumerate}[\rm(i)]
	\item $\langle [\theta Y_{\lambda}, \xi], Z \rangle_{B_\theta} =\langle [\theta Y_{\lambda}, X_{\alpha}], Z \rangle_{B_\theta} =  0$ for all $Z \in \g{a} \oplus \g{n}$. \label{lemma:projection:xi:y:i}
	\item $\langle[ \theta [Y_{\lambda}, \xi], \xi], Z \rangle_{B_\theta} = 0$ for all $Z \in \g{a} \oplus \g{n}$. \label{lemma:projection:xi:y:ii}
	\item If $\lambda\neq \alpha$, then $[Y_{\lambda}, \theta \xi]$, $[[Y_{\lambda}, \xi], \theta \xi]$, $[Y_{\lambda}, \theta X_{\alpha}] \in \g{n}$.\label{lemma:projection:xi:y:iii} 
\end{enumerate}
\end{lemma}
\begin{proof}
First, using the properties of the root space decomposition, we have $[\theta Y_{\lambda}, \xi] = a [ \theta Y_{\lambda}, H_{\alpha}] + b[ \theta Y_{\lambda}, X_{\alpha}] \in \g{g}_{-\lambda} \oplus \g{g}_{\alpha-\lambda}$. Since $\lambda$ is a positive root by assumption, then $-\lambda$ is negative, and hence $\langle \g{g}_{-\lambda},\g{a}\oplus\g{n}\rangle_{B_\theta}=0$. Moreover, since $\alpha\in \Pi$ is a simple root and $\lambda\in\Sigma^+$, then $\alpha-\lambda\notin\Sigma^+$, and hence either $\lambda=\alpha$ or $\g{g}_{\alpha-\lambda}=0$. If $\lambda=\alpha$, then $[\theta Y_{\lambda}, X_{\alpha}] \in \g{k}_0$ by  Lemma~\ref{lemma:berndt:sanmartin}~(\ref{lemma:berndt:sanmartin:ii}). In any case, $ [\theta Y_{\lambda}, X_{\alpha}]$ is $\langle \cdot,\cdot\rangle_{B_\theta}$-orthogonal to $\g{a}\oplus\g{n}$. This proves claim~(\ref{lemma:projection:xi:y:i}).

Now, we have
\[
[\theta [Y_{\lambda}, \xi], \xi] = [\theta [Y_{\lambda}, aH_\alpha+b X_\alpha], \xi]= -a \langle \alpha, \lambda \rangle [\theta Y_{\lambda}, \xi] + b [\theta[Y_{\lambda}, X_{\alpha}], \xi].
\]
Note first that $ [\theta Y_{\lambda}, \xi]$ is $\langle \cdot,\cdot\rangle_{B_\theta}$-orthogonal to $\g{a} \oplus \g{n}$ by virtue of assertion~(\ref{lemma:projection:xi:y:i}). Moreover, by the properties of the root spaces, we have $[\theta[Y_{\lambda}, X_{\alpha}], \xi] \in \g{g}_{-\lambda-\alpha} \oplus \g{g}_{-\lambda}$, where $-\lambda-\alpha$, $-\lambda\notin\Sigma^+\cup\{0\}$. This implies assertion~(\ref{lemma:projection:xi:y:ii}).

Finally, we prove claim~(\ref{lemma:projection:xi:y:iii}). Again by the properties of the root space decomposition, we have $[Y_{\lambda}, \theta \xi] \in \g{g}_{\lambda} \oplus \g{g}_{\lambda - \alpha}$ and $[[Y_{\lambda}, \xi], \theta \xi] \in \g{g}_{\lambda} \oplus \g{g}_{\lambda-\alpha} \oplus \g{g}_{\lambda + \alpha}$. Since by assumption $\lambda$ is a positive root different from $\alpha$, we have  $\lambda + \alpha\in \Sigma^+$ or $\lambda + \alpha\notin \Sigma\cup\{0\}$, and $\lambda - \alpha\in\Sigma^+$ or $\lambda - \alpha\notin\Sigma\cup\{0\}$. In both cases, we get $[Y_{\lambda}, \theta \xi]$, $[[Y_{\lambda}, \xi], \theta \xi] \in \g{n}$, and taking $a=0$ and $b=1$ in the first bracket, we also have $[Y_{\lambda}, \theta X_\alpha]\in\g{n}$, from where the result follows.
\end{proof}

Our purpose hereafter will be to calculate the restriction of the endomorphism $D$ defined in~\eqref{eq:operator:d} to the subspace $\g{g}_\beta\oplus\g{g}_{\alpha+\beta}$ of $\g{s}$ (see~\eqref{eq:decomposition:tangent}). This will be done in several steps. In the following proposition we restrict our attention to the operator $R_{\xi} + \Ss_{\xi}^2$. Note that, since $\beta\in\Sigma^+\setminus\{\alpha\}$ and $\alpha+\beta\in\Sigma^+$ by definition of $\beta$, for both $\lambda=\beta$ and $\lambda=\alpha+\beta$, we have that $\lambda\neq\alpha$, and hence $\g{g}_\lambda\subset\g{s}$. Therefore, the following proposition implies that $R_{\xi} + \Ss_{\xi}^2\vert_{\g{g}_\beta\oplus\g{g}_{\alpha+\beta}}=0$.

\begin{proposition}\label{proposition:r:s:n}
Let $\xi = a H_{\alpha} + b X_{\alpha}$ be a unit vector, where $\alpha \in \Pi$, $X_{\alpha} \in \g{g}_{\alpha}$ is a unit vector, and $a$, $b \in \R$. Let $Y_{\lambda} \in \g{g}_{\lambda}\subset \g{s}$, where $\lambda\in\Sigma^{+}\setminus\{\alpha\}$. Then $(R_{\xi} + \Ss_{\xi}^2) Y_{\lambda} = 0$.
\end{proposition}

\begin{proof}
We will use several times that $\alpha-\lambda\notin\Sigma^+$, as follows from $\alpha\in \Pi$ and $\lambda\in\Sigma^{+}$.
First, using Lemma~\ref{lemma:xi:calculations}~(\ref{lemma:xi:calculations:i})-(\ref{lemma:xi:calculations:ii}), one has 
\[ 
\nabla_{\xi} \xi \in \g{a} \oplus \g{g}_\alpha , \qquad \nabla_{\xi} \xi - \theta \nabla_{\xi} \xi = [ \theta \xi , \xi ] . 
\] 
Hence, by 
Lemma~\ref{lemma:levi:civita:connection:n} 
(taking into account that $\langle Y_{\lambda}, \g{g}_{\alpha} \rangle = 0$ since $\lambda\neq\alpha$),  we have 
\begin{align}\label{eq:r:s:1}
\begin{split} 
\nabla_{Y_{\lambda}} \nabla_{\xi} \xi & = 
\frac{1}{2} ([Y_{\lambda}, \nabla_{\xi} \xi] + [\theta Y_{\lambda} , \nabla_{\xi} \xi] - [Y_{\lambda}, \theta \nabla_{\xi} \xi])_{\g{n}} 
= \frac{1}{2} [Y_{\lambda}, [\theta \xi, \xi]] . 
\end{split} 
\end{align} 
Now, using again Lemma~\ref{lemma:levi:civita:connection:n}, Lemma~\ref{lemma:projection:xi:y}~(\ref{lemma:projection:xi:y:i})-(\ref{lemma:projection:xi:y:iii}), and the symmetry of the Levi-Civita connection, we get
\begin{equation}\label{eq:r:s:2}
\begin{aligned}
 \nabla_{\xi} \nabla_{Y_{\lambda}} \xi & =  \frac{1}{2} \nabla_\xi ([Y_{\lambda}, \xi] + [\theta Y_{\lambda}, \xi] - [Y_{\lambda}, \theta \xi])_\g{n} = \frac{1}{2} \nabla_\xi [Y_{\lambda}, \xi] - \frac{1}{2} \nabla_\xi [Y_{\lambda}, \theta \xi]  \\
 & = \frac{1}{2} [\xi, [Y_{\lambda}, \xi]] + \frac{1}{2} \nabla_{[Y_{\lambda}, \xi]} \xi - \frac{1}{2} \nabla_\xi [Y_{\lambda}, \theta \xi].
\end{aligned}
\end{equation}
As $\lambda\in\Sigma^+\setminus\{\alpha\}$, we have that $[Y_{\lambda}, \xi]\in\g{g}_\lambda\oplus\g{g}_{\alpha+\lambda}$ and $\xi\in\g{a}\oplus\g{g}_\alpha$ are under the hypotheses of Lemma~\ref{lemma:levi:civita:connection:n}. Hence, using Lemma~\ref{lemma:levi:civita:connection:n} together with Lemma~\ref{lemma:projection:xi:y}~(\ref{lemma:projection:xi:y:ii})-(\ref{lemma:projection:xi:y:iii}), we have
\begin{align}\label{eq:r:s:3}
\nabla_{[Y_{\lambda}, \xi]} \xi & = \frac{1}{2} ([[Y_{\lambda}, \xi], \xi] + [\theta [Y_{\lambda}, \xi], \xi] - [[Y_{\lambda}, \xi],\theta \xi])_{\g{n}} = \frac{1}{2} [[Y_{\lambda}, \xi], \xi] - \frac{1}{2}[[Y_{\lambda}, \xi],\theta \xi].
\end{align}
Again, using Lemma~\ref{lemma:levi:civita:connection:n}, Lemma~\ref{lemma:projection:xi:y}~(\ref{lemma:projection:xi:y:i})-(\ref{lemma:projection:xi:y:iii}), the symmetry of the Levi-Civita connection, and the Jacobi identity, we obtain
\begin{equation}\label{eq:r:s:4}
\begin{aligned}
\Ss_{\xi}^2 Y_{\lambda} & = \nabla_{\nabla_{Y_{\lambda}} \xi} \xi = \frac{1}{2} \nabla_{([Y_{\lambda}, \xi]+[\theta Y_{\lambda}, \xi]-[Y_{\lambda}, \theta\xi])_{\g{n}}} \xi = \frac{1}{2} \nabla_{[Y_{\lambda}, \xi]} \xi - \frac{1}{2} \nabla_{[Y_{\lambda}, \theta \xi]} \xi\\
& = \frac{1}{2} \nabla_{[Y_{\lambda}, \xi]} \xi + \frac{1}{2}[\xi,[Y_{\lambda}, \theta \xi]] -\frac{1}{2} \nabla_\xi [Y_{\lambda}, \theta \xi]\\
& = \frac{1}{2} \nabla_{[Y_{\lambda}, \xi]} \xi - \frac{1}{2}[Y_{\lambda},[\theta \xi, \xi]]  - \frac{1}{2}[\theta \xi,[\xi, Y_{\lambda}]]  -\frac{1}{2} \nabla_\xi [Y_{\lambda}, \theta \xi].
\end{aligned}
\end{equation}
Now, using~\eqref{eq:r:s:1}, \eqref{eq:r:s:2}, \eqref{eq:r:s:3} and \eqref{eq:r:s:4} we get
\begin{align*}
(R_{\xi} + \Ss_{\xi}^2) Y_{\lambda} & =  \nabla_{Y_{\lambda}} \nabla_{\xi} \xi - \nabla_{\xi} \nabla_{Y_{\lambda}} \xi - \nabla_{[Y_{\lambda}, \xi]} \xi + \Ss_{\xi}^2 Y_{\lambda} \\
& =   \frac{1}{2}[Y_{\lambda}, [\theta \xi, \xi]] - \frac{1}{2} [\xi, [Y_{\lambda}, \xi]] - \frac{1}{2} \nabla_{[Y_{\lambda}, \xi]} \xi 
+ \frac{1}{2} \nabla_\xi [Y_{\lambda}, \theta \xi] -\frac{1}{2} [[Y_{\lambda}, \xi], \xi]\\
& \phantom{=} + \frac{1}{2}[[Y_{\lambda}, \xi],\theta \xi]+ \frac{1}{2} \nabla_{[Y_{\lambda}, \xi]} \xi - \frac{1}{2}[Y_{\lambda},[\theta \xi, \xi]] 
- \frac{1}{2}[\theta \xi,[\xi, Y_{\lambda}]]  -\frac{1}{2} \nabla_\xi [Y_{\lambda}, \theta \xi] \\
& =  0.   \qedhere
\end{align*}
\end{proof}

In view of Proposition~\ref{proposition:r:s:n}, in order to determine the restriction $D\vert_{\g{g}_{\beta} \oplus \g{g}_{\alpha + \beta}}$ of the endomorphism $D$ defined in~\eqref{eq:operator:d} to the subspace $\g{g}_{\beta} \oplus \g{g}_{\alpha + \beta}$, we just have to calculate the restriction of the shape operator $\Ss_\xi$ to $\g{g}_{\beta} \oplus \g{g}_{\alpha + \beta}$. 
Before doing that, we introduce a result which will make easier the calculations concerning the shape operator.

\begin{lemma}\label{lemma:unitary:brackets}
Let $Y_{\lambda} \in \g{g}_{\lambda}$ and $X_{\alpha} \in \g{g}_{\alpha}$ be unit vectors, where $\lambda \in \Sigma^{+}$ and $\alpha \in \Pi$. Assume $\lambda - \alpha\notin\Sigma$ and $\lambda + \alpha\in\Sigma$. Then
\[
[Y_{\lambda + \alpha}, \theta X_{\alpha}]  = - |\alpha| \sqrt{- A_{\alpha, \lambda} } Y_{\lambda}, 
\]
where $Y_{\lambda + \alpha} :=[Y_{\lambda},  X_{\alpha}]/(|\alpha| \sqrt{- A_{\alpha, \lambda} })$
 is a unit vector in $\g{g}_{\lambda + \alpha}$.
\end{lemma}

\begin{proof}
First, using \eqref{eq:relation:inner} taking into account $[Y_{\lambda}, X_{\alpha}]$, $Y_{\lambda} \in \g{n}$, \eqref{eq:cartan:inner}, the Jacobi identity, the assumption $\lambda - \alpha\notin\Sigma$, Lemma~\ref{lemma:berndt:sanmartin}~(\ref{lemma:berndt:sanmartin:i}), and the definition of the Cartan integers $A_{\alpha,\lambda}$, we obtain
\begin{align*}
\langle [Y_{\lambda}, X_{\alpha}], [Y_{\lambda}, X_{\alpha}] \rangle & = \frac{1}{2} \langle [Y_{\lambda}, X_{\alpha}], [Y_{\lambda}, X_{\alpha}]\rangle_{B_\theta}  = \frac{1}{2} \langle Y_{\lambda}, [\theta X_{\alpha}, [Y_{\lambda}, X_{\alpha}]] \rangle_{B_\theta} \\
& = -\frac{1}{2} \langle Y_{\lambda}, [Y_{\lambda}, [ X_{\alpha}, \theta X_{\alpha}]] + [X_{\alpha}, [\theta X_{\alpha}, Y_{\lambda}]] \rangle_{B_\theta}\\
&  = - \langle \alpha, \lambda \rangle  \langle Y_{\lambda}, Y_{\lambda} \rangle_{B_\theta} = -|\alpha|^2 A_{\alpha, \lambda} \langle Y_{\lambda}, Y_{\lambda} \rangle = -|\alpha|^2 A_{\alpha, \lambda}.
\end{align*}
By the standard theory of abstract root systems (see for example \cite[Proposition~2.48~(g)]{K}) we have $A_{\alpha,\lambda}< 0$, as $\lambda - \alpha\notin\Sigma$ and $\lambda + \alpha\in\Sigma$. Hence $Y_{\lambda + \alpha} =[Y_{\lambda},  X_{\alpha}]/(|\alpha| \sqrt{- A_{\alpha, \lambda} })\in\g{g}_{\lambda+\alpha}$ is a unit vector.  Now, using the Jacobi identity, the fact that $\lambda - \alpha\notin \Sigma$, and Lemma~\ref{lemma:berndt:sanmartin}~(\ref{lemma:berndt:sanmartin:i}), we get 
\[
[[Y_{\lambda}, X_{\alpha}], \theta X_{\alpha}] = - [[\theta X_{\alpha}, Y_{\lambda}], X_{\alpha}] - [[X_{\alpha}, \theta X_{\alpha}], Y_{\lambda}] = 2 \langle \alpha, \lambda \rangle  Y_{\lambda} = |\alpha|^2 A_{\alpha, \lambda} Y_\lambda.
\]
Thus, recalling the definition of the unit vector $Y_{\lambda + \alpha}$, we have $[Y_{\lambda + \alpha}, \theta X_{\alpha}] = -|\alpha| \sqrt{-A_{\alpha, \lambda}} Y_{\lambda}$. This finishes the proof.
\end{proof}

From now on, for each $\lambda \in \Sigma^{+}$ such that $\lambda+\alpha\in\Sigma$ and $\lambda - \alpha\notin\Sigma$ (e.g.\ for $\lambda=\beta$), and for a unit vector $Y_{\lambda} \in \g{g}_{\lambda}$, we define
\begin{equation}\label{eq:definition:unit}
Y_{\lambda + \alpha} := \frac{[Y_{\lambda}, X_{\alpha}]}{|\alpha| \sqrt{- A_{\alpha, \lambda} }}\in\g{g}_{\lambda+\alpha},
\end{equation}
which is a unit vector in $\g{g}_{\lambda + \alpha}$ by means of Lemma~\ref{lemma:unitary:brackets}. Note that $Y_{\lambda + \alpha}$ depends on the choice of $Y_\lambda$, but we remove this dependence from the notation for simplicity. We can now proceed with the calculation of the shape operator.

\begin{proposition}\label{proposition:shape:operator:n}
Let $Y_{\lambda} \in \g{g}_{\lambda}$ be a unit vector, for some $\lambda \in \Sigma^{+}\setminus\{\alpha\}$. Then, 
\begin{equation}\label{eq:general:shape:operator}
\Ss_{\xi} Y_{\lambda} = \frac{a}{2} |\alpha|^2 A_{\alpha, \lambda} Y_{\lambda} - \frac{b}{2} [Y_{\lambda}, X_{\alpha}] + \frac{b}{2}[Y_{\lambda}, \theta X_{\alpha}]. 
\end{equation}
In particular, if $\lambda \in \Sigma^{+}\setminus\{\alpha\}$ is such that $\lambda + \alpha\in\Sigma$  but $\lambda - \alpha\notin\Sigma$, then:
\begin{align}
	\Ss_{\xi} Y_{\lambda} &= \frac{a}{2} |\alpha|^2 A_{\alpha, \lambda} Y_{\lambda} - \frac{b}{2} |\alpha|\sqrt{-A_{\alpha, \lambda}}  Y_{\lambda + \alpha}.  \label{proposition:shape:operator:n:1}
	\\
	\Ss_{\xi} Y_{\lambda+\alpha} &= \frac{a }{2} |\alpha|^2 A_{\alpha, \lambda + \alpha} Y_{\lambda + \alpha} + Y_{\lambda + 2 \alpha}  -\frac{b}{2}|\alpha|\sqrt{-A_{\alpha, \lambda}}  Y_{\lambda},\label{proposition:shape:operator:n:2}
\end{align}
for some $Y_{\lambda + 2 \alpha}\in\g{g}_{\lambda + 2 \alpha}$. (If $\lambda + 2\alpha\notin\Sigma$ then $Y_{\lambda + 2 \alpha} = 0$.)  
\end{proposition} 
\begin{proof}
First, using Lemma~\ref{lemma:levi:civita:connection:n} and Lemma~\ref{lemma:projection:xi:y}~(\ref{lemma:projection:xi:y:i})-(\ref{lemma:projection:xi:y:iii}), we get
\begin{align*}
\Ss_{\xi} Y_{\lambda} & = -\nabla_{Y_{\lambda}} \xi = -\frac{1}{2}([Y_{\lambda}, \xi] + [\theta Y_{\lambda}, \xi] - [Y_{\lambda}, \theta \xi])_\g{n} = -\frac{1}{2}[Y_{\lambda}, \xi]+\frac{1}{2}[Y_{\lambda}, \theta \xi]\\
& = -\frac{1}{2}[Y_{\lambda}, a H_{\alpha} + bX_{\alpha}]+\frac{1}{2}[Y_{\lambda}, -a H_{\alpha} + b\theta X_{\alpha}]
 =  a \langle \alpha, \lambda \rangle Y_{\lambda} - \frac{b}{2} [Y_{\lambda}, X_{\alpha}] + \frac{b}{2}[Y_{\lambda}, \theta X_{\alpha}]\\
& = \frac{a}{2} |\alpha|^2 A_{\alpha, \lambda} Y_{\lambda} - \frac{b}{2} [Y_{\lambda}, X_{\alpha}] + \frac{b}{2}[Y_{\lambda}, \theta X_{\alpha}].
\end{align*}
This proves~\eqref{eq:general:shape:operator}. Note that \eqref{proposition:shape:operator:n:1} follows directly from~\eqref{eq:general:shape:operator}, taking into account \eqref{eq:definition:unit} and $\lambda - \alpha\notin\Sigma$. Finally, using~\eqref{eq:general:shape:operator} and Lemma~\ref{lemma:unitary:brackets} we get
\begin{align*}
\Ss_{\xi} Y_{\lambda + \alpha} & =\frac{a}{2}|\alpha|^2 A_{\alpha, \lambda + \alpha} Y_{\lambda + \alpha} - \frac{b}{2} [Y_{\lambda+ \alpha}, X_{\alpha}] + \frac{b}{2}[Y_{\lambda + \alpha}, \theta X_{\alpha}]\\
& = \frac{a}{2} |\alpha|^2 A_{\alpha, \lambda + \alpha} Y_{\lambda + \alpha} + Y_{\lambda + 2 \alpha} - \frac{b}{2} |\alpha|\sqrt{-A_{\alpha, \lambda}} Y_{\lambda},
\end{align*}
where we put  $Y_{\lambda + 2 \alpha} := - b[Y_{\lambda+ \alpha}, X_{\alpha}]/2\in \g{g}_{\lambda + 2 \alpha}$. This proves~\eqref{proposition:shape:operator:n:2}.
\end{proof}
Therefore, taking into account the definition of the endomorphism $D$ in~\eqref{eq:operator:d}, Proposition~\ref{proposition:r:s:n} and Proposition~\ref{proposition:shape:operator:n}, we can state the following:

\begin{corollary}\label{corollary:d:n}
Let $\xi = a H_{\alpha} + b X_{\alpha}$ be a unit vector, where $X_{\alpha} \in \g{g}_{\alpha}$ is a unit vector, $\alpha \in \Pi$, and $a$, $b \in \R$. Let $\lambda \in \Sigma^{+}\backslash \{ \alpha \}$ be such that $\lambda + \alpha\in\Sigma$ and $\lambda - \alpha\notin\Sigma$. Let $Y_{\lambda}\in \g{g}_{\lambda}\subset\g{s}$ be a unit vector. Then there exists $Y_{\lambda + 2 \alpha}\in\g{g}_{\lambda + 2 \alpha}$ such that:
\begin{align*}
D Y_{\lambda} &= \tr (\Ss_{\xi}) \biggl(\frac{a}{2}|\alpha|^2 A_{\alpha, \lambda} Y_{\lambda} - \frac{b}{2}|\alpha| \sqrt{-A_{\alpha, \lambda}} Y_{\lambda + \alpha}\biggr) + c Y_{\lambda}, 
\\
D Y_{\lambda + \alpha}  &= \tr (\Ss_{\xi}) \biggl(\frac{a}{2} |\alpha|^2 A_{\alpha, \lambda+ \alpha} Y_{\lambda + \alpha} + Y_{\lambda + 2 \alpha} - \frac{b}{2} |\alpha|\sqrt{-A_{\alpha, \lambda}} Y_{\lambda}\biggr) + c Y_{\lambda+\alpha}. 
\end{align*}
\end{corollary}

Corollary~\ref{corollary:d:n} with $\lambda=\beta$ gives us the expression of $D\vert_{\g{g}_\beta\oplus\g{g}_{\beta+\alpha}}$. Thus, so far, we have calculated $D$ restricted to the subspaces $\g{a} \ominus \R H_{\alpha}$ and $\g{g}_{\beta} \oplus \g{g}_{\beta + \alpha}$ of $\g{s}$. Hence, calculating $DU$ is the last step in order to determine the endomorphism $D$ restricted to the subspace~\eqref{eq:decomposition:tangent} of $\g{s}$. Recall that $U = b |\alpha|^{-1} H_{\alpha} -a |\alpha| X_{\alpha}$ is a unit vector in $\g{s}$. Let us first introduce the following auxiliary result.

\begin{lemma}\label{lemma:u}
We have:
	\begin{tasks}[label=(\roman*),label-width=4ex,label-format={\rm}](3)
		\task $\nabla_{X_{\alpha}} \xi = |\alpha| U,$
		\label{lemma:u:1}
		\task $\nabla_U \xi = -a |\alpha|^2 U$,\label{lemma:u:2}
		\task $\nabla_U U = a |\alpha|^2 \xi$,\label{lemma:u:4}
		\task $\nabla_\xi U = -b |\alpha| \xi$,\label{lemma:u:5}
		\task $[U, \xi] = |\alpha| X_{\alpha}$,\label{lemma:u:6}
		\task $\nabla_\xi \xi = b |\alpha| U$. \label{lemma:u:7}
	\end{tasks}
\setcounter{enumi}{0}
\end{lemma}
\begin{proof}
In this proof, let $Z\in\g{a}\oplus\g{n}$ be arbitrary. Firstly, using~\eqref{eq:inner:an:bphi},  Lemma~\ref{lemma:berndt:sanmartin}~(\ref{lemma:berndt:sanmartin:i}) and~\eqref{eq:relation:inner}, we have
\begin{align*} 
\langle \nabla_{X_{\alpha}} \xi, Z \rangle & = \frac{1}{4}\langle [X_{\alpha}, a H_{\alpha} + b X_{\alpha}] + [\theta X_{\alpha}, a H_{\alpha} + b X_{\alpha}] - [X_{\alpha}, \theta (a H_{\alpha} + b X_{\alpha})], Z  \rangle_{B_\theta}\\
& = \frac{1}{4} \langle -a [H_{\alpha}, X_{\alpha}] + b [\theta X_{\alpha}, X_{\alpha}] - a [H_{\alpha}, X_{\alpha}] -b [X_{\alpha}, \theta X_{\alpha}], Z  \rangle_{B_\theta} \\
& =  -a |\alpha|^2 \frac{1}{2} \langle X_{\alpha}, Z \rangle_{B_\theta} + b \langle   H_{\alpha}, Z  \rangle_{B_\theta}
 = -a |\alpha|^2 \langle X_{\alpha}, Z \rangle + b\langle   H_{\alpha}, Z  \rangle =\langle |\alpha|U,Z\rangle.
\end{align*}
This proves claim~\ref{lemma:u:1}. Using  Lemma~\ref{lemma:xi:calculations}~(\ref{lemma:xi:calculations:iii}) and claim~\ref{lemma:u:1} we get
\[
\nabla_U \xi = b|\alpha|^{-1}  \nabla_{H_{\alpha}} \xi - a |\alpha| \nabla_{X_{\alpha}} \xi = -a |\alpha|^2 U,
\] 
which proves claim~\ref{lemma:u:2}. Now, using again~\eqref{eq:inner:an:bphi}, Lemma~\ref{lemma:berndt:sanmartin}~(\ref{lemma:berndt:sanmartin:i}) and~\eqref{eq:relation:inner}, we obtain
\begin{align*}\label{calculation:levi:xalpha:u}
\nonumber\langle \nabla_{X_{\alpha}} U, Z \rangle & = \frac{1}{4}\langle [X_{\alpha},U] + [\theta X_{\alpha}, U] - [X_{\alpha}, \theta U], Z  \rangle_{B_\theta}\\
\nonumber& = \frac{1}{4} \langle -b |\alpha|^{-1}[H_{\alpha}, X_{\alpha}] -a |\alpha| [\theta X_{\alpha}, X_{\alpha}] -b |\alpha|^{-1}[H_{\alpha}, X_{\alpha}] + a |\alpha| [X_{\alpha}, \theta X_{\alpha}] , Z  \rangle_{B_\theta} \\
\nonumber& = - a|\alpha| \langle H_{\alpha}, Z \rangle_{B_\theta} -\frac{1}{2}b |\alpha| \langle X_{\alpha}, Z \rangle_{B_\theta}  = -a|\alpha| \langle H_{\alpha}, Z \rangle  -b |\alpha| \langle X_{\alpha}, Z \rangle.
\end{align*}
Hence $\nabla_{X_{\alpha}} U=- |\alpha|\xi$, and then, using Lemma~\ref{lemma:xi:calculations}~(\ref{lemma:xi:calculations:iii}), we get
\[
\nabla_U U = b|\alpha|^{-1}  \nabla_{H_{\alpha}} U - a |\alpha| \nabla_{X_{\alpha}} U = a |\alpha|^2 \xi,
\qquad
\nabla_\xi U = a \nabla_{H_{\alpha}} U + b \nabla_{X_{\alpha}} U = -b |\alpha| \xi,
\]
which proves \ref{lemma:u:4} and \ref{lemma:u:5}. Now, taking into account that $a^2 |\alpha|^2 + b^2 = 1$ since $\xi$ is a unit vector, we get
\[
[U, \xi] = [b |\alpha|^{-1} H_{\alpha} -a |\alpha| X_{\alpha}, a H_{\alpha} + b X_{\alpha}] = b^2 |\alpha| X_{\alpha} + a^2 |\alpha|^3 X_{\alpha} =  |\alpha| X_{\alpha},
\]
which proves assertion~\ref{lemma:u:6}. Finally, claim~\ref{lemma:u:7} follows directly from Lemma~\ref{lemma:xi:calculations}~(\ref{lemma:xi:calculations:ii}).
\end{proof}
\begin{proposition}\label{proposition:d:u}
$D U = (\tr (\Ss_{\xi}) a |\alpha|^2 + b^2 |\alpha|^2 + c)U$.
\end{proposition}

\begin{proof}
Using Lemma~\ref{lemma:u} several times, and the relation $a^2 |\alpha|^2 -1 = -b^2$, we get
\begin{align*}
(R_{\xi} + \Ss_{\xi}^2) U & = \nabla_U \nabla_\xi \xi - \nabla_\xi \nabla_U \xi - \nabla_{[U, \xi]} \xi + \nabla_{\nabla_U \xi} \xi\\
& =b |\alpha|  \nabla_U U +a |\alpha|^2 \nabla_\xi U - |\alpha|\nabla_{X_{\alpha}} \xi -a |\alpha|^2 \nabla_U \xi \\
& = ab |\alpha|^3 \xi -ab |\alpha|^3 \xi -|\alpha|^2U  + a^2 |\alpha|^4 U = -b^2 |\alpha|^2 U.
\end{align*}
Moreover, by Lemma~\ref{lemma:u}~\ref{lemma:u:2} we have $\Ss_{\xi} U = - \nabla_{U} \xi = a |\alpha|^2 U$. Altogether:
\[
D U = (\tr (\Ss_{\xi}) \Ss_{\xi} - (R_{\xi} + \Ss_{\xi}^2) + c \id) U = (\tr (\Ss_{\xi}) a |\alpha|^2  + b^2 |\alpha|^2  + c)U.\qedhere
\]
\end{proof}

Proposition~\ref{proposition:d:a}, Corollary~\ref{corollary:d:n} and Proposition~\ref{proposition:d:u} provide an explicit description of the endomorphism $D$ defined in~\eqref{eq:operator:d} restricted to the subspace~\eqref{eq:decomposition:tangent} of $\g{s}$. We are therefore in position to prove Theorem~\ref{th:xi_diagonal}.

\begin{proof}[Proof of Theorem~\ref{th:xi_diagonal}]
As above in this section, we assume that $M$ has rank greater than one, $\g{s} =  (\g{a} \oplus \g{n}) \ominus \R \xi$ is a Lie subalgebra  of $\g{a}\oplus\g{n}$ with $\xi = a H_{\alpha} + b X_{\alpha}$ of unit length, for some $\alpha \in \Pi$, some unit vector $X_{\alpha} \in \g{g}_{\alpha}$, and $a$, $b \in \R$, and $S$ is an algebraic Ricci soliton, that is, the endomorphism $D$ given in~\eqref{eq:operator:d} is a derivation of $\g{s}$. The result will be proved if we show that $b = 0$. 	
	
The first step will be to show $c = 0$, where $c$ is the scalar involved in \eqref{eq:operator:d}. In order to do so, select a root $\lambda \in \Sigma^{+} \setminus \{ \alpha \}$ such that $\lambda + \alpha\in\Sigma$ and $\lambda - \alpha\notin\Sigma$. Recall that such a root exists since $M$ is irreducible and of rank greater than one by assumption: it~suffices to choose a simple root $\lambda$ connected to $\alpha$ in the Dynkin diagram of $M$. Now take a unit vector $Y_{\lambda} \in \g{g}_{\lambda}\subset\g{s}$. Since $\lambda$ is not proportional to $\alpha$ and $\dim\g{a}\geq 2$, there exists an element $H \in \g{a} \ominus \R H_{\alpha}\subset\g{s}$ such that $\lambda (H) =\langle H,H_\lambda\rangle\neq 0$. As $D$ is assumed to be a derivation,~then  
\begin{equation}\label{eq:derivation:a}
D[H, Y_{\lambda}] = [DH, Y_{\lambda}] + [H, DY_{\lambda}].
\end{equation}
Note that $D[H, Y_{\lambda}] = \lambda(H) D Y_{\lambda}$. Moreover, by Proposition~\ref{proposition:d:a}, we have $[DH, Y_{\lambda}] = c \lambda(H) Y_{\lambda}$. Since by assumption $\langle H, H_{\alpha} \rangle = 0$, then $(\lambda + \alpha) H = \lambda(H)$, and hence by Corollary~\ref{corollary:d:n} we get $[H, DY_{\lambda}] = \lambda(H) D Y_{\lambda}$. Altogether, \eqref{eq:derivation:a} now becomes $\lambda(H) D Y_{\lambda} = c \lambda(H)  Y_{\lambda} + \lambda(H) D Y_{\lambda}$. Since $\lambda(H) \neq 0$ and $Y_{\lambda}$ is a non-zero vector, we get $c = 0$, as desired.

In order to conclude that $b=0$, we will examine the relation $D[U, Y_{\lambda}] = [DU, Y_{\lambda}] + [U, DY_{\lambda}]$, which holds because $D$ is a derivation. In particular, we have
\begin{equation}\label{eq:derivation:u:y}
\langle D[U, Y_{\lambda}], Y_{\lambda} \rangle = \langle[DU, Y_{\lambda}] , Y_{\lambda} \rangle+ \langle[U, DY_{\lambda}], Y_{\lambda} \rangle.
\end{equation}
Recalling from~\eqref{eq:definition:unit} the definition of $Y_{\lambda+\alpha}$, we get
\begin{align}\label{eq:bracket:u:ylambda}
 [U, Y_{\lambda}] & = \bigl[b |\alpha|^{-1} H_{\alpha} -a |\alpha| X_{\alpha}, Y_{\lambda}\bigr] = \dfrac{b |\alpha|}{2} A_{\alpha, \lambda} Y_{\lambda} +a \sqrt{-A_{\alpha, \lambda}} |\alpha|^2 Y_{\lambda + \alpha}.
\end{align}
Now, using \eqref{eq:bracket:u:ylambda} and Corollary~\ref{corollary:d:n} along with $c=0$, we obtain 
\begin{equation}\label{eq:derivation:u:y:1}
\begin{aligned}
\langle D[U, Y_{\lambda}], Y_{\lambda} \rangle & =  \dfrac{b |\alpha|}{2} A_{\alpha, \lambda} \langle DY_{\lambda}, Y_{\lambda} \rangle+a \sqrt{-A_{\alpha, \lambda}} |\alpha|^2 \langle DY_{\lambda + \alpha}, Y_{\lambda} \rangle\\
&=  \tr(\Ss_{\xi}) \dfrac{ab}{4} |\alpha|^3 A_{\alpha, \lambda}^2 + \tr(\Ss_{\xi})\dfrac{ ab}{2} |\alpha|^3 A_{\alpha, \lambda}.
\end{aligned}
\end{equation}
From Proposition~\ref{proposition:d:u} and $c=0$ we get $DU=(a\tr  (\Ss_\xi)+b^2)|\alpha|^2 U$, and then by \eqref{eq:bracket:u:ylambda} we~have
\begin{equation}\label{eq:derivation:u:y:2}
\langle[DU, Y_{\lambda}] , Y_{\lambda} \rangle  =  (a\tr (\Ss_\xi)+b^2)|\alpha|^2 \langle [U, Y_{\lambda}], Y_{\lambda} \rangle 
 = \frac{b}{2}(a\tr (\Ss_\xi)+b^2)|\alpha|^3 A_{\alpha, \lambda}.
\end{equation}
Corollary~\ref{corollary:d:n}, \eqref{eq:bracket:u:ylambda}, and the fact that $[U, Y_{\lambda +  \alpha}] \in \g{g}_{\lambda + \alpha} \oplus \g{g}_{\lambda + 2 \alpha}$ yield
\begin{equation}\label{eq:derivation:u:y:3}
\begin{aligned}
\langle[U, DY_{\lambda}], Y_{\lambda} \rangle & =  \frac{a}{2}\tr (\Ss_{\xi}) |\alpha|^2A_{\alpha, \lambda} \langle [U, Y_{\lambda}], Y_{\lambda} \rangle - \frac{b}{2}\tr (\Ss_{\xi}) \sqrt{-A_{\alpha, \lambda}}|\alpha| \langle [U, Y_{\lambda + \alpha}], Y_{\lambda} \rangle \\
& = \tr (\Ss_{\xi}) \frac{ab}{4} |\alpha|^3 A_{\alpha, \lambda}^2.
\end{aligned}
\end{equation}
Finally, inserting \eqref{eq:derivation:u:y:1}, \eqref{eq:derivation:u:y:2} and \eqref{eq:derivation:u:y:3} into ~\eqref{eq:derivation:u:y}, we get $ b^3  |\alpha|^3  A_{\alpha, \lambda} =0$. As $\lambda - \alpha\notin\Sigma$ and $\lambda + \alpha\in\Sigma$, we have $A_{\alpha, \lambda} \neq 0$ (see~\cite[Proposition~2.48 (g)]{K}). Since of course $|\alpha| \neq 0$ as $\alpha\in\Sigma$, we conclude that $b=0$, as we wanted to show.
\end{proof}

We can now put several results together to conclude the proof of Theorem~\ref{th:symmetric}.

\begin{proof}[Proof of Theorem~\ref{th:symmetric}]
	Let $S$ be a codimension one subgroup of the solvable part $AN$ of an Iwasawa decomposition of a symmetric space of non-compact type $M$. Let us write $\g{s}=(\g{a}\oplus\g{n})\ominus\R\xi$ for the Lie algebra of $S$. First note that, since $\g{a}\oplus\g{n}$ is completely solvable, the same is true for the Lie subalgebra $\g{s}=(\g{a}\oplus\g{n})\ominus\R\xi$. Hence, by~\cite{Lauret:crelle}, the connected, simply connected Lie subgroup $S$ of the Iwasawa group $AN$ with Lie algebra $\g{s}$ is a Ricci soliton if and only if it is an algebraic Ricci soliton. 
	
	By Lemma~\ref{lemma:JDG}, if $S$ is an algebraic Ricci soliton, then either $\xi\in\g{a}$, or $\xi\notin \g{a}$ and $\xi\in\R H_\alpha\oplus\g{g}_\alpha$ for some $\alpha\in\Pi$. In the first case $S$ is a connected Lie subgroup of $AN$ containing $N$; and, conversely, any connected Lie subgroup of $AN$ containing $N$ is a solvsoliton by virtue of Proposition~\ref{prop:horosphere}. 
	
	In the second case, Theorem~\ref{th:xi_diagonal} implies that $M$ has rank one. Therefore, $M$ is a hyperbolic space and, hence,  homothetic to a Damek-Ricci space or to a real hyperbolic space $\R \mathsf{H}^n$, as recalled in~\S\ref{subsec:DR_prelim}. In the first subcase, the solvable part $AN$ of an Iwasawa decomposition associated with $M$ has the structure of a Damek-Ricci space, up to homothety. Then Theorem~\ref{th:DR} implies that $S$ is an algebraic Ricci soliton if and only if $S$ is either the subgroup $N$ of $AN$ (that is, a horosphere in the rank one symmetric space $M$) or a Lohnherr hypersurface $W^3$ in a complex hyperbolic plane $\C \mathsf{H}^2$. In the second subcase, if $M$ is a real hyperbolic space $\R \mathsf{H}^n$, the isometric action of $S$ on $\R \mathsf{H}^n$ is of cohomogeneity one and without singular orbits (as $S$ is a codimension one subgroup of $AN$, which acts simply transitively on $\R\mathsf{H}^n$). Then any of the orbits of such action is isoparametric and, by the theory of isoparametric hypersurfaces in space forms (see for example~\cite[\S3.1]{CR:book}), it must indeed be a totally umbilic hypersurface with constant principal curvatures. Hence, by the Gauss equation of submanifold geometry, such orbits always have constant sectional curvature and are therefore Einstein. 
\end{proof}

\end{document}